\documentclass[11pt,leqno]{amsart}
\usepackage{amssymb}
\usepackage{amsfonts}
\usepackage{amsthm}
\usepackage{mathtools}
\usepackage{todonotes}

\calclayout

\usepackage{amsthm}

\usepackage[pagebackref,colorlinks,citecolor=blue,linkcolor=blue]{hyperref}

\usepackage{cite}   

\usepackage{amsmath,bbm,amssymb,amsxtra}

\def\Xint#1{\mathchoice
	{\XXint\displaystyle\textstyle{#1}}
	{\XXint\textstyle\scriptstyle{#1}}
	{\XXint\scriptstyle\scriptscriptstyle{#1}}
	{\XXint\scriptscriptstyle\scriptscriptstyle{#1}}
	\!\int}

\def\XXint#1#2#3{{\setbox0=\hbox{$#1{#2#3}{\int}$ }
		\vcenter{\hbox{$#2#3$ }}\kern-.6\wd0}}
\def\dashint{\Xint-}

\newtheorem{theorem}{Theorem}[section]
\newtheorem{lemma}[theorem]{Lemma}
\newtheorem{proposition}[theorem]{Proposition}
\newtheorem{remark}[theorem]{Remark}
\newtheorem{corollary}[theorem]{Corollary}
\newtheorem{definition}[theorem]{Definition}
\newtheorem{example}[theorem]{Example}
\newtheorem{question}[theorem]{Question}
\newtheorem{open}[theorem]{Open problem}
\numberwithin{equation}{section}

\title{Existence and uniqueness of limits at infinity for bounded variation functions}

\author{Panu Lahti}
\address{Academy of Mathematics and Systems Science, Chinese Academy of Sciences, Beijing 100190, PR China}
\email{panulahti@amss.ac.cn}

\author{Khanh Nguyen}
\address{Academy of Mathematics and Systems Science, Chinese Academy of Sciences, Beijing 100190, PR China}
\email{khanhnguyen@amss.ac.cn}


\begin{document}
	\maketitle
	
	\begin{abstract}
		In this paper, we study  the existence of limits at infinity along almost every infinite curve for the upper and lower approximate limits of  bounded variation functions on complete unbounded metric measure spaces. We prove that if the measure is  doubling and supports a $1$-Poincar\'e inequality, then for every bounded variation function $f$ and for $1$-a.e. infinite curve $\gamma$, for both the upper approximate limit $f^\vee$ and the lower approximate limit $f^\wedge$ we have that 
		\[
		\lim_{t\to+\infty}f^\vee(\gamma(t)) {\rm \ \ and\ \ }\lim_{t\to+\infty}f^\wedge(\gamma(t))
		\]
		exist and are equal to the same finite value. We give examples showing that the conditions of the doubling property of the measure and a $1$-Poincar\'e inequality are needed for the existence of limits. 
		Furthermore, we establish a characterization for strictly positive $1$-modulus of the family of all infinite curves in terms of bounded variation functions.
		These generalize  results for Sobolev functions given in \cite{KN23}.
	\end{abstract}

	\tableofcontents
	\section{Introduction}
	\subsection{Existence of limits at infinity}\ 
	
	Let us first recall a result from \cite{KN23}. Under the assumption that $X=(X,d,\mu)$ is a doubling metric measure space supporting a $p$-Poincar\'e inequality and satisfying the annular chain condition,
	with $1\leq p<+\infty$, Theorem 1.1 in \cite{KN23} gives that for every $f\in \dot N^{1,p}(X)$, there is a constant $c_f\in\mathbb R$ so that the limit 
	\begin{equation}
		\label{eq1.1-1511} \lim_{t\to+\infty}f(\gamma(t)) \text{\rm \ \ exists and is equal to $c_f$ for  $p$-a.e. infinite curve $\gamma\in \Gamma^{+\infty}$.}
	\end{equation}
	By infinite curves $\gamma$ we mean that $\gamma\setminus B\neq\emptyset$ for all balls $B$;
	the  notions of doubling measure, Poincar\'e inequality, the homogeneous $p$-Sobolev spaces $\dot N^{1,p}(X)$, etc.
	are all given in Section \ref{section2}.
	For each function $f$, if $\eqref{eq1.1-1511}$ holds true then we say that the existence and uniqueness of limits at infinity
	hold for $f$ along $p$-a.e. infinite curve.
	
	The aim of this paper is to extend the result \eqref{eq1.1-1511} to functions of bounded variation. For a subset $A\subseteq X$, let ${\rm BV}(A)$ be the collection of all integrable functions on $A$ for which the total variation \eqref{2.8} is finite.
	We write $f\in {\rm BV}_{\rm loc}(X)$ if for each $x\in X$ there is a ball $B(x,r)$ with center $x$ and radius $r>0$ so that  $f\in {\rm BV}(B(x,r))$. We set $\dot{\rm BV}(X)$ to be the collection of all locally integrable functions $f\in {\rm BV}_{\rm loc}(X)$ for which the total variation $\|D f\|(X)$
	is finite.
	We will study the existence and uniqueness of  the  limits
	\begin{equation}
		\label{eq1.1-2808} \lim_{t\to+\infty} f^\vee(\gamma(t)){\rm \ \ and\ \ } \lim_{t\to+\infty} f^\wedge(\gamma(t))  
	\end{equation}
	for $\gamma\in\Gamma^{+\infty}$ and
	$f\in \dot{\rm BV}(X)$.  Here  $f^\vee(x)$ and $f^\wedge(x)$ are the upper and lower approximate limits of $f$ at $x$ as defined in \eqref{eq2.6-1511}-\eqref{equ2.12-0912}.
	If $f\in \dot N^{1,1}(X)$ then the existence and uniqueness of limits at infinity  of $f^\vee$ and $f^\wedge$
	along $1$-a.e. infinite curve  hold because of Proposition \ref{prop2.2-1511} and
	\eqref{eq:Mod and H}.
	
	In \cite{KN23}, from the definition of weak upper gradients as given in \eqref{def-upper-gradient} 
	it is easy to get the existence of the limits \eqref{eq1.1-2808} for $f\in \dot N^{1,1}(X)$ and for $1$-a.e. infinite curve,
	since  $f\in \dot N^{1,1}(X)$ is absolutely continuous on   $1$-a.e. infinite curve $\gamma$.
	Then we are interested in the following question:
	\begin{question}\label{question1}
		Let $f\in \dot{\rm BV}(X)$.
		\begin{center}
			Does the existence of the limits \eqref{eq1.1-2808} hold true for ${\rm AM}$-a.e. infinite curve $\gamma$?
		\end{center}
		Here the ${\rm AM}$-modulus is defined in \eqref{eq2.13-1711}.
	\end{question}
	Firstly, we give the following theorem to understand Question \ref{question1}. 
	\begin{theorem}\label{theorem1.2}
		Suppose that $(X,d,\mu)$ is an unbounded metric measure space.
		Let $f\in \dot{\rm BV}(X)$. Then for ${\rm AM}$-a.e. infinite curve $\gamma$ parameterized by arc-length, there exist $c_\gamma\in\mathbb R$ and a subset $N_{\gamma}\subset[0,+\infty)$ with $\mathcal L^1(N_\gamma)=0$ such that both limits
		\begin{equation}
			\label{1.3-8thMar}\notag
			\lim_{t\to+\infty, t\notin N_\gamma}f^\vee(\gamma(t)) {\rm \ \ and \ \ } \lim_{t\to+\infty, t\notin N_\gamma}f^\wedge(\gamma(t))
		\end{equation}
		exist and are equal to $c_\gamma$.
		Here $\mathcal L^1$ is the $1$-Lebesgue measure on $[0,+\infty)$.
	\end{theorem}
	
	By Examples \ref{ex3.3-0912}, \ref{ex3.4-1612}, \ref{ex3.5-8thMar}, \ref{ex3.6-30th03}, \ref{ex3.7-30thMar},  the conclusion of Theorem \ref{theorem1.2} is sharp in that the unbounded null set $N_\gamma$  cannot be taken to be empty
	in general. 
	We construct a space that satisfies neither the doubling condition nor a $1$-Poincar\'e inequality in   Examples \ref{ex3.3-0912}-\ref{ex3.4-1612}, a space that supports a $1$-Poincar\'e inequality but the doubling condition fails in Example \ref{ex3.5-8thMar}, and a space that is doubling but does not support a $1$-Poincar\'e inequality in Examples \ref{ex3.6-30th03}-\ref{ex3.7-30thMar}, so that  limits at infinity of $f^\vee, f^\wedge$ along  infinite curves $\gamma$ may fail to exist for ${{\rm AM}}$-a.e. $\gamma$. 
	Thus the answer is ``NO" in general to Question \ref{question1}.
	
	Now, we give a positive answer to Question \ref{question1} under the conditions of the doubling property of the measure and the support of a $1$-Poincar\'e inequality, as our next main theorem below.
	
	\begin{theorem}\label{theorem1.5}
		Suppose that $(X,d,\mu)$ is a complete doubling unbounded metric measure space supporting a $1$-Poincar\'e inequality.
		Let $f\in \dot{\rm BV}(X)$. 
		Then for  ${1}$-a.e. $\gamma\in\Gamma^{+\infty}$, there exists $c_\gamma\in\mathbb R$ such that  the limits
		\[
		\lim_{t\to+\infty}f^\vee(\gamma(t)) {\rm \ \ and\ \ } \lim_{t\to+\infty}f^\wedge(\gamma(t)) 
		\]
		exist and are equal to $c_\gamma$. In particular, 
		both limits exist for ${\rm AM}$-a.e. $\gamma\in\Gamma^{+\infty}$.
	\end{theorem}
	
	For the uniqueness part  in the case of $\dot{\rm BV}$-functions, we  obtain it by using the same arguments as in \cite{KN23} in the case of $\dot{N}^{1,1}$-functions, under the assumption of the annular chain condition as defined in Definition \ref{dyadic}, see the details in Section \ref{sec3}.
	In the following corollary, we also prove that the limit at infinity of ${\rm BV}$-functions along $1$-a.e. infinite curve
	is necessarily $0$.

	\begin{corollary}\label{thm1.1}
		Suppose that $(X,d,\mu)$ is a complete doubling unbounded metric measure space supporting a $1$-Poincar\'e inequality and satisfying an annular $\lambda$--chain condition, where $\lambda\ge 1$ is the scaling factor in the Poincar\'e inequality.
		Then for every $f\in \dot {\rm BV}(X)$, there exists $c_f\in\mathbb R$ so that 
		\[
		\lim_{t\to+\infty} f^\vee (\gamma(t)) {\rm \ \ and\ \ } \lim_{t\to+\infty}f^\wedge(\gamma(t))
		\]
		exist and are equal to $c_f$ for $1$-a.e. infinite curve $\gamma\in\Gamma^{+\infty}$.
		
		Moreover, if we additionally assume that $f\in {\rm BV}(X)$ then $c_f= 0$.
		
		In particular, the claim holds  true for ${\rm AM}$-a.e. infinite curve.
	\end{corollary}

	The assumption of the annular chain condition is  necessary in Corollary \ref{thm1.1}. For example, on a space with at least two ends one easily constructs a Lipschitz function in $\dot {\rm BV}(X)$ (or even $\dot {N}^{1,1}(X)$) so that the tail value  of two of the ends is a different constant.
	A tree \cite{NW20,BBGS17} has more than two ends and is a doubling  metric measure space supporting a $1$-Poincar\'e inequality. Moreover, the assumption of the doubling property of the measure and a Poincar\'e inequality is also necessary for the uniqueness properties by Remark \ref{rem3.5-1612}.
	
	
	Apart from the representatives $f^\wedge$ and $f^\vee$, one can also consider other representatives of BV functions, and their
	limits at infinity.
	
	\begin{remark}\label{rem1.5-12nd4}
		Let $h$ be a representative of $f\in\dot{\rm BV}(X)$ so that $h(x)$ is ``between" $f^\vee(x)$ and $f^\wedge(x)$ in the sense that
		\begin{equation}
			\label{1.4-8thMar}
			af^\vee(x)+bf^\wedge(x)\leq h(x)\leq cf^\vee(x)+df^\wedge(x)
		\end{equation}
		holds true for $\mathcal H$-a.e. $x$ and for some finite constants $a,b,c,d,$ with $a+b=c+d=1$. Here $\mathcal H$ is the codimension $1$ Hausdorff measure defined in \eqref{2.7-8thMarch}. For instance, we have from \eqref{2.26-25Feb}  that the Lebesgue representative $f^*$ and the average approximation $\bar f:= \frac{f^\vee+f^\wedge}{2}$   satisfy  the property \eqref{1.4-8thMar}.  Here $f^*$ means the Lebesgue representative of $f$ defined in \eqref{eq-lebes}.

		We obtain from
		\eqref{eq:Mod and H}
		that for $1$-a.e. infinite curve $\gamma$, if $\lim_{t\to+\infty}f^\vee(\gamma(t))=\lim_{t\to+\infty}f^\wedge(\gamma(t))=c_\gamma$  then the limit at infinity along $\gamma$ of the representative $h$ satisfying \eqref{1.4-8thMar}  exists and is equal to $c_\gamma$. In particular, Theorem \ref{theorem1.5}  gives that the doubling and $1$-Poincar\'e inequalities condition is sufficient for the existence of limits at infinity along $1$-a.e. infinite curve for the Lebesgue representative $f^*$ and the average approximation $\bar f$ of $f\in \dot{\rm BV}(X)$.
	\end{remark}
	The condition of a $1$-Poincar\'e inequality is needed for the existence of limits at infinity of the Lebesgue representative $f^*$ of a $\dot{\rm BV}$-function $f$ for $1$-a.e. infinite curve, by Example \ref{ex3.4-1612} (the doubling condition fails) and by Examples \ref{ex3.6-30th03}-\ref{ex3.7-30thMar} (the doubling condition holds).
	However, we do not know whether the doubling condition is necessary, and so we pose the following open problem.
	\begin{open}
		Suppose that $(X,d,\mu)$ supports a $1$-Poincar\'e inequality. 
		Let $f\in \dot{\rm BV}(X)$.
		\begin{center}
			Does the existence of $\lim_{t\to+\infty}f^*(\gamma(t))$ for $1$-a.e. infinite curve $\gamma$  hold?
		\end{center}
	\end{open}

	\subsection{Strictly positive $1$-modulus of family of all infinite curves}\ 
	
	The conclusions of Theorem \ref{theorem1.5} and Corollary \ref{thm1.1} are nontrivial only when the $1$-modulus of $\Gamma^{+\infty}$ is strictly positive. 
	In \cite{KN23},  the  $1$-modulus of the family $\Gamma^{+\infty}$ being strictly positive is shown to be equivalent to the  quantity 
	\[
	\mathcal R_1:=\sup_{j\in\mathbb N}\frac{2^j}{\mu(B(O,2^{j+1})\setminus B(O,2^j))}
	\]
	being finite, and also equivalent to the claim that for every $f\in \dot{N}^{1,1}(X)$, there exists an infinite curve $\gamma$ so that  $\lim_{t\to+\infty}f(\gamma(t))$ exists under the assumption of $X$ being a complete, doubling unbounded metric measure space supporting a $1$-Poincar\'e inequality and the annular chain condition. Here $B(O,r)$ is the ball with radius $r$ and center at a given point $O\in X$.  For instance, on an Ahlfors  $Q$-regular complete unbounded  metric measure space supporting a $1$-Poincar\'e inequality, we  always have ${\rm Mod}_1(\Gamma^{+\infty})>0$ where $1<Q<+\infty$ and   Ahlfors $Q$-regular measures are defined in \eqref{eq2.4-1611}. We denote $N^{1,1}(X)=\dot N^{1,1}(X)\cap L^1(X)$.
	Next, we give more characterizations of ${\rm Mod}_1(\Gamma^{+\infty})>0$ in terms of functions of bounded variation. 
	
	\begin{theorem}\label{thm1.2}
		Under the assumption of Corollary \ref{thm1.1}, the following statements are equivalent:
		\begin{enumerate}
			\item ${\rm Mod}_1(\Gamma^{+\infty})>0$.
			\item For every $f\in \dot{\rm BV}(X)$, there exists an infinite curve $\gamma\in\Gamma^{+\infty}$ so that both limits
			\[
			\lim_{t\to+\infty}f^\vee(\gamma(t)) {\rm \ \ and \ \ }\lim_{t\to+\infty}f^\wedge(\gamma(t))
			\]
			exist and are equal.
			\item For every $f\in \dot{\rm BV}(X)\setminus \dot{ N}^{1,1}(X)$, there exists an infinite curve $\gamma\in\Gamma^{+\infty}$ so that both limits   \[
			\lim_{t\to+\infty}f^\vee(\gamma(t)) {\rm \ \ and \ \ }\lim_{t\to+\infty}f^\wedge(\gamma(t))
			\]
			exist and are equal.
			\item For every $f\in { N}^{1,1}(X)$, there exists an infinite curve $\gamma\in\Gamma^{+\infty}$ so that  $$\lim_{t\to+\infty}f(\gamma(t))=0.$$
			\item    For every $f\in {\rm BV}(X)\setminus \dot{ N}^{1,1}(X)$, there exists an infinite curve $\gamma\in\Gamma^{+\infty}$ so that both limits   \[
			\lim_{t\to+\infty}f^\vee(\gamma(t)) {\rm \ \ and \ \ }\lim_{t\to+\infty}f^\wedge(\gamma(t))
			\]
			exist and are equal to $0$.
			\item For every $f\in {\rm BV}(X)$, there exists an infinite curve $\gamma\in\Gamma^{+\infty}$ so that both limits  \[
			\lim_{t\to+\infty}f^\vee(\gamma(t)) {\rm \ \ and \ \ }\lim_{t\to+\infty}f^\wedge(\gamma(t))
			\]
			exist and are equal to $0$.
			\item For every $f\in \dot N^{1,1}(X)\setminus N^{1,1}(X)$, there exists an infinite curve $\gamma\in \Gamma^{+\infty}$ so that 
			\[
			\lim_{t\to+\infty}f(\gamma(t))
			\]
			exist.
		\end{enumerate}
		
	\end{theorem}

	A reader familiar with classification theory should recognize that the quantity $\mathcal R_1$ is finite in \cite{G99,H99,HK01, K22} and that ${\rm Mod}_1(\Gamma^{+\infty})$ is strictly positive in \cite{Sh21} as a condition towards $1$-hyperbolicity. In \cite{KN23, EKN22}, the finiteness of $\mathcal R_1$ gives that for every subset $E$ of the unit sphere $\mathbb S$ on Muckenhoupt $\mathcal A_1$-weighted spaces  or on polarizable Carnot groups,  $\sigma(E)\lesssim {\rm Mod}_1(\Gamma_E)$ where $\sigma$ is the probability surface measure on $\mathbb S$ and $\Gamma_E$ is the collection of all radial curves starting from a point in $E$. It follows from Corollary \ref{thm1.1}  and Lemma \ref{lem4.1-1809} that the finiteness of $\mathcal R_1$ characterizes the existence and uniqueness of  limits at infinity along almost every radial curve for bounded variation functions on both Muckenhoupt $\mathcal A_1$-weighted spaces and polarizable Carnot groups, i.e. $\mathcal R_1<+\infty$ is equivalent to the fact that for every $f\in\dot{\rm BV}(X)$, there is $c_f\in\mathbb R$ so that $\lim_{r\to+\infty}f(r\cdot\xi)=c_f$ for $\sigma$-a.e. $\xi\in\mathbb S$.


	\section{Preliminaries} \label{section2}
	
	Throughout this paper, we use the following conventions.
	We work in a metric measure space $(X,d,\mu)$, where $\mu$ is a Borel regular outer measure.
	We denote by $O$ the given point in our space $X$.
	The notation $A\lesssim B\ (A\gtrsim B)$ means that there is a constant $C>0$ depending only on the data such that $A\leq C \cdot B\ (A\geq C\cdot B)$, and $A\approx B$ (or $A\simeq B$) means that $A\lesssim B$ and $A\gtrsim B$. 
	For  each locally integrable function $f$ and for every measurable subset $A\subset X$ of strictly positive
	and finite measure, we let $f_A:=\dashint_Afd\mu=\frac{1}{\mu(A)}\int_Afd\mu$.  We denote the ball centered at $x\in X$ of radius $r>0$ by $B(x,r)$, and for $\tau>0$, we set $\tau B(x,r):=B(x,\tau r)$. We always assume that $\mu(X)>0$.
	\subsection{Infinite curves}
	Let $(X,d)$ be a metric space. A  \textit{curve} is a nonconstant continuous mapping from an interval $I\subseteq\mathbb R$ into $X$. The \textit{length} of a curve $\gamma$ is denoted by $\ell(\gamma)$. A curve $\gamma$ is said to be a  \textit{rectifiable curve} if its length is finite. Similarly, $\gamma$ is a   \textit{locally rectifiable curve} if its restriction to each closed subinterval of $I$ is rectifiable. Each rectifiable curve $\gamma$ will be parameterized by arc length and hence the  \textit{line integral} over $\gamma$ of a Borel function $f$ on $X$ is 
	\[\int_\gamma fds =\int_0^{\ell(\gamma)}f(\gamma(t))dt.
	\]
	If $\gamma$ is locally rectifiable, then we set 
	\[\int_{\gamma}fds=\sup\int_{\gamma'}fds
	\]where the supremum is taken over all rectifiable subcurves $\gamma'$ of $\gamma$. Let $\gamma:[0,{+\infty})\to X$ be a locally rectifiable curve, parameterized by arc length. Then
	\[\int_{\gamma}fds=\int_0^{+\infty} f(\gamma(t))dt.
	\]
	A locally rectifiable curve $\gamma$ is an  \textit{infinite curve} if $\gamma\setminus B\neq\emptyset$ for all balls $B$. Then $\int_{\gamma}ds={+\infty}.$   
	We denote by $\Gamma^{+\infty}$ the collection of all infinite curves.

	\subsection{Modulus, capacity and Newtonian spaces}\label{sec2.2}
	Let $\Gamma$ be a family of curves in  $(X,d,\mu)$. Let $1\le p<+\infty$. The  \textit{$p$-modulus} of $\Gamma$, denoted  $\text{\rm Mod}_p(\Gamma)$, is defined by 
	\[\text{\rm Mod}_p(\Gamma):=\inf \int_{X}\rho^pd\mu
	\]where the infimum is taken over all Borel functions $\rho:X\to[0,{+\infty}]$ satisfying 
	$\int_{\gamma}\rho ds\geq 1
	$ for every locally rectifiable curve $\gamma\in\Gamma$. A family of curves is called \textit{$p$-exceptional} if it has $p$-modulus zero. We say that a property  holds for \textit{$p$-a.e. curve} (or $p$-almost every curve) if the collection of curves for which the property fails  is $p$-exceptional.

	Let $u$ be a locally integrable function on $X$. A Borel function $\rho:X\to[0,{+\infty}]$ is said to be an \textit{upper gradient} of $u$ if 
	\begin{equation}
		\label{def-upper-gradient}|u(x)-u(y)|\leq \int_{\gamma}\rho\, ds
	\end{equation}
	for every rectifiable curve $\gamma$ connecting $x$ and $y$. Then we  have that \eqref{def-upper-gradient} holds for all compact subcurves of $\gamma\in\Gamma^{+\infty}$. For $1\le p<+\infty$, we say that $\rho$ is a  \textit{$p$-weak upper gradient} of $u$ if \eqref{def-upper-gradient} holds for $p$-a.e. rectifiable curve. If $u$ has a $p$-weak upper gradient that is $p$-integrable, we denote by $g_u$ the  \textit{minimal $p$-weak} upper gradient of $u$, which is unique up to sets of measure zero and which is minimal in the sense that $g_u\leq\rho $ a.e.\ for every $p$-integrable $p$-weak upper gradient $\rho$ of $u$. In \cite{H03}, the existence and uniqueness of such a minimal $p$-weak upper gradient are given.  The notion of upper gradients is due to Heinonen and Koskela \cite{HK98}, and we refer  interested  readers to \cite{HKST15,BB15,H03,N00,HK98} for a more detailed discussion on upper gradients.
	
	Let $\dot{N}^{1,p}(X)$ be the collection of all locally integrable functions for which an upper gradient is $p$-integrable on $X$ where $1\leq p<+\infty$. The space $\dot N^{1,p}(X)$ is called the homogeneous $p$-Sobolev space on $(X,d,\mu)$.
	
	The \textit{$p$-capacity} of $K\subset X$ 
	is defined by 
	\[\text{\rm Cap}_p(K):=\int_{X}|u|^pd\mu+\inf \int_{X}g_u^pd\mu,
	\] 
	where the infimum is taken over all $p$-integrable functions $u:X\to\mathbb R$ with $p$-integrable minimal {$p$-weak} upper gradient $g_u$ such that $u|_K\equiv 1$. 
	{By \cite[Lemma 6.2.2]{HKST15}, one can replace the minimal $p$-weak upper gradient $g_u$ in the definition
		with upper gradients of the function $u$.}

	
	Let $\Gamma$ be a family of locally rectifiable curves on $X$. We define the ${\rm AM}$-modulus of $\Gamma$ on $X$ by setting 
	\begin{equation}\label{eq2.13-1711}
		{\rm AM}(\Gamma):=\inf \liminf_{i\to+\infty} \int_X g_{i}d\mu
	\end{equation}
	where the infimum is taken over all sequences of Borel functions $g_i:X\to[0,+\infty]$  so that 
	\[
	\liminf_{i\to+\infty}\int_{\gamma}g_ids \geq 1
	\]
	for all $\gamma\in\Gamma$. We have that ${\rm AM}(\Gamma)\leq {\rm Mod}_1(\Gamma)$ for all families $\Gamma$ of locally rectifiable curves. We say that a property holds on ${\rm AM}$-a.e. curve if it holds for every curve except for a family $\Gamma$ with ${\rm AM}(\Gamma)=0.$ 
	We refer the interested reader to \cite{HKMO21,Mar16} for a discussion on ${\rm AM}$-modulus.

	\begin{lemma}\label{lem2.2-2508}
		Let $\Gamma$ be a family of locally rectifiable curves. Then the following are equivalent:
		\begin{enumerate}
			\item ${\rm AM}(\Gamma)=0$.
			\item For every $\varepsilon>0$, 
			there is a sequence of functions, denoted $\{ \rho_i\}_{i\in\mathbb N}$, such that 
			\[
			\sup_{i\in\mathbb N}   \int_X\rho_id\mu <\varepsilon
			\]
			and that
			\[
			\liminf_{i\to+\infty}\int_{\gamma}\rho_ids=+\infty {\rm \ \ for\ all\ }\gamma\in\Gamma.
			\]
		\end{enumerate}
	\end{lemma}
	\begin{proof}\ 
		
		$(2)\Longrightarrow(1)$ is immediate.
		It suffices to show $(1)\Longrightarrow(2)$. Let ${\rm AM}(\Gamma)=0$. Let $\varepsilon>0$. Then for each $j$, we find a sequence of admissible functions $\rho_{i,j}$ satisfying
		\[
		\liminf_{i\to+\infty}\int_\gamma\rho_{i,j}ds\geq 1
		\]
		for all $\gamma\in\Gamma$ so that
		\[
		\liminf_{i\to+\infty} \int_X\rho_{i,j}d\mu\leq \frac{1}{2^j} \frac{\varepsilon}{100}.
		\]
		We set $\rho_i:=\sum_{j\in\mathbb N}\rho_{i,j}$. By taking a subsequence, still denoted $\rho_{i,j}$, we may assume that  
		\[
		\int_X\rho_{i,j}d\mu\leq \frac{1}{2^j}\frac{\varepsilon}{10} {\rm \ \ for \ all \ }i\in\mathbb N.
		\]
		Then 
		\[
		\liminf_{i\to+\infty}\int_\gamma \rho_ids=\liminf_{i\to+\infty}\sum_{j\in\mathbb N}\int_\gamma\rho_{i,j}ds\geq  \sum_{j\in\mathbb N}\liminf_{i\to+\infty} \int_\gamma\rho_{i,j}ds \geq \sum_{j\in\mathbb N}1=+\infty 
		\]
		for all $\gamma\in\Gamma$ and
		\[
		\sup_{i\in\mathbb N}   \int_X\rho_{i}d\mu= \sup_{i\in\mathbb N}\sum_{j\in\mathbb N}\int_X \rho_{i,j}d\mu \leq \sum_{j\in\mathbb N} \frac{1}{2^j}\frac{\varepsilon}{10}<\varepsilon.
		\]
		The proof completes.
	\end{proof}

	\subsection{Doubling measures and Poincar\'e inequalities}
	We always assume every ball in $X$ to have strictly positive and finite $\mu$-measure.
	The Borel regular outer measure $\mu$ is called \textit{doubling} if  there exists a constant $C_d\geq 1$ such that for all balls $B(x,r)$  with radius $r>0$ and center at $x\in X$,
	\begin{equation}\label{eq2.4-0912}
		\mu(B(x,2r))\leq C_d\mu(B(x,r)).
	\end{equation}
	Here $C_d$ is called the \textit{doubling constant}.
	
	Let $1< Q<{+\infty}$. The Borel regular outer measure $\mu$ is said to be \textit{Ahlfors $Q$-regular} if there exists a constant $C_Q\geq 1$ such that for all balls $B(x,r)$ with radius $r>0$ and center at $x\in X$,
	\begin{equation}\label{eq2.4-1611}
		\frac{r^Q}{C_Q}\leq \mu(B(x,r))\leq C_Qr^Q.
	\end{equation}
	If $\mu$ is Ahlfors $Q$-regular where $1<Q<{+\infty}$, then $\mu$ is a doubling measure.

	Let $1\leq p<{+\infty}$. We say that $X$, or $\mu$, supports a  \textit{$p$-Poincar\'e inequality}
	if there exist constants $C>0$ and $\lambda\geq 1$ such that 
	\begin{equation}\label{eq2.5-2711}
		\dashint_{B(x,r)}|u-u_{B(x,r)}|d\mu \leq C r\left (\dashint_{B(x,\lambda r)}\rho^pd\mu\right )^{1/p}
	\end{equation}
	for all balls $B(x,r)$ with radius $r>0$ and center at $x\in X$, and for all pairs $(u,\rho)$ satisfying \eqref{def-upper-gradient} such that $u$ is integrable on balls. Here $\lambda$ is called the {\it scaling constant} or {\it scaling factor} of the $p$-Poincar\'e inequality. 
	For more on Poincar\'e inequalities, we refer the interested reader to \cite{HK00,HK95,HKST15,Hei01}.

	\subsection{Hausdorff measures}
	Let $0\le \beta<{+\infty}$ and $0<R\leq{+\infty}$. The {\it $(\beta,R)$-Hausdorff content}  of a subset $E$, denoted  $\mathcal H^\beta_R(E)$, is given by
	\begin{equation}
		\notag
		\mathcal H^\beta_R(E)=\inf\left\{\sum_{k\in\mathbb N} r_k^\beta: E\subset \bigcup_{k\in\mathbb N}B_k\text{\rm \ \ and \ \ } 0<r_k<R \right\}
	\end{equation}
	where $B_k$ is a ball with radius $r_k$. The {\it $\beta$-Hausdorff measure} $\mathcal H^\beta$ of a subset $E$ is 
	\[
	\mathcal H^{\beta}(E):=\lim_{R\to 0}\mathcal H^\beta_R(E).
	\]
	When the measure $\mu$ is only assumed to be doubling, and not necessarily Ahlfors $Q$-regular, it is more natural to define the  \emph{codimension $1$ Hausdorff content} of a subset $E$, for  $0<R\le+\infty$, as follows:
	\[
	{\mathcal H}_R(E)=\inf\left\{\sum_{k\in\mathbb N}\frac{\mu(B_k)}{r_k}:E\subset\bigcup_{k\in\mathbb N}B_k\quad\text{and}\quad 0<r_k<R\right\}.
	\]
	The \emph{codimension $1$ Hausdorff measure} of a subset $E$ is then defined by 
	\begin{equation}\label{2.7-8thMarch}
		{\mathcal H}(E):=\lim_{R\to 0}{\mathcal H}_R(E).
	\end{equation}
	If $\mu$ is assumed to be Ahlfors $Q$-regular, then ${\mathcal H}_R\simeq{\mathcal H}^{Q-1}_R$ and ${\mathcal H}\simeq{\mathcal H}^{Q-1}$.
	
	For $A\subset X$,
	let $\Gamma_A$ be the family of all infinite curves
	$\gamma\in\Gamma^{+\infty}$ for which $\gamma\cap A\neq \emptyset$. 
	For any admissible sequence $\{ B_k\}_{k\in\mathbb N}$ for computing $\mathcal H_R(A)$, we have that the function $\sum_{k\in\mathbb N}\frac{\chi_{2B_k}}{r_k}$ is admissible for computing ${\rm Mod}_1(\Gamma_A)$ where $2B_k$ is the ball with radius $2r_k$ and the same center as $B_k$. Then we obtain that
	\begin{equation}
		\label{eq:Mod and H}
		{\rm Mod}_1(\Gamma_A)\leq C_d \mathcal H(A)
	\end{equation}
	where $C_d$ is the doubling constant.

	\subsection{Bounded variation functions}\label{sec2.5}
	
	In this subsection, we recall the definition and basic properties of functions of bounded variation on metric spaces as in \cite{Mi03}, see also the monographs for the classical theory in the Euclidean setting \cite{AFP00,EG15,Fe69,Gi84,Zi89}. 
	
	Let $f\in L^1_{\rm loc}(X)$ be  a locally function on $X$. Given an open subset $\Omega\subseteq X$, we define the total variation of $f$ on $\Omega$ by 
	\begin{equation}\label{eq2.7-2711}\notag
		\| Df\|(\Omega):=\inf\left\{ \liminf_{i\to+\infty}\int_\Omega g_{f_i}d\mu: f_i\in N^{1,1}_{\rm loc}(\Omega), f_i\to f {\rm\ in\ }L^1_{\rm loc}(\Omega)\right\}
	\end{equation}
	where each $g_{f_i}$ is the minimal $1$-weak upper gradient of $f_i$ in $\Omega$. In \cite{Mi03}, local Lipschitz constants were used in place of upper gradients but the theory can be developed similarly with either definition. We say that a function $f\in L^1(\Omega)$ is of bounded variation, denoted $f\in {\rm BV}(\Omega)$, if the total variation  $\|Df\|(\Omega)$ of $f$ is finite. For an arbitrary subset $A\subset X$, we define 
	\begin{equation}\label{2.8}
		\| Df\|(A):=\inf \{\| Df\|(W): A\subset W, W\subset X {\rm \ is\ open} \}.
	\end{equation}
	By \cite[Theorem 3.4]{Mi03}, $\|Df\|$ is a Radon measure on a given open subset $\Omega$ for $f\in L^1_{\rm loc}(\Omega)$ and $\|Df\|(\Omega)<+\infty$. 
	For open subsets $A\subset X$, we say that $f\in {\rm BV}_{\rm loc}(A)$ if for each $x\in A$, there is a ball $B(x,r)$ with center at $x$ and radius $r>0$ so that $f\in{\rm BV}(B(x,r))$. If $X$ is proper, then $f\in{\rm BV}_{\rm loc}(A)$ is equivalent to having  $f\in {\rm BV}(A')$ for all open $A'$ such that $\overline{A'}$ is a compact
	subset of $A$. Here $X$ is called proper if every bounded and closed set is compact. For instance, from \cite[Lemma 4.1.14]{HKST15}, every complete doubling metric space is proper.
	We say that $f\in\dot{\rm BV}(A)$ if $f\in {\rm BV}_{\rm loc}(A)$ and $\|Df\|(A)<+\infty$.
	
	By the definition of total bounded variation of a function, a $1$-Poincar\'e inequality \eqref{eq2.5-2711} gives that the  inequality
	\begin{equation}\label{eq2.5-1611}
		\dashint_{B(x,r)}|f-f_{B(x,r)}|d\mu\leq C r \frac{\|Df\|(\lambda B(x,r))}{\mu(\lambda B(x,r))}
	\end{equation}
	holds true for all balls $B(x,r)$ and for every $f\in \dot {\rm BV}_{\rm loc}(X)$
	where $C, \lambda$ are the constant and scaling factor of $1$-Poincar\'e inequality respectively.

	A $\mu$-measurable set $E\subset X$ is said to be of finite perimeter if $\|D\chi_E\|(X)<+\infty$ where $\chi_E$ is the characteristic function of $E$. The perimeter of $E$ in $\Omega$ is denoted by 
	\[
	P(E,\Omega):=\|D\chi_E\|(\Omega).
	\]

	The ${\rm BV}$ norm and $\dot{\rm BV}$ seminorm of $f$ on $\Omega\subset X$ are respectively defined by 
	\[
	\| f\|_{\rm BV(\Omega)}:= \| f\|_{L^1(\Omega)} + \|Df\|(\Omega)
	\]
	and
	\[
	\|f\|_{\dot{\rm BV}(\Omega)}:=\|Df\|(\Omega).
	\]
	
	The measure-theoretic interior $I_E$ and exterior $O_E$ of a set $E\subset X$ are defined respectively by
	\[
	I_E:=\left\{x\in X:  \lim_{r\to 0} \frac{\mu(B(x,r)\setminus E)}{\mu(B(x,r))}=0 \right\}
	\]
	and
	\[
	O_E:=\left\{x\in X: \lim_{r\to 0} \frac{\mu(B(x,r)\cap E)}{\mu(B(x,r))}=0 \right\}.
	\]
	The measure-theoretic boundary $\partial^*E$ is defined as the set of points $x\in X$ at which both $E$ and its complement have strictly positive upper density, i.e.
	\[
	\limsup_{r\to 0} \frac{\mu(B(x,r)\setminus E)}{\mu(B(x,r))}>0\ \ {\rm and\ \ }  \limsup_{r\to 0} \frac{\mu(B(x,r)\cap E)}{\mu(B(x,r))}>0.
	\]
	Then $X=I_E\cup O_E\cup\partial^*E$.

	Given an open subset $\Omega\subset X$ and a $\mu$-measurable subset $E\subset X$ with $P(E,\Omega)<+\infty$, we have that for any Borel set $A\subset \Omega$,
	\begin{equation}
		P(E,A)=\int_{\partial^*E\cap A}\theta_Ed\mathcal H \label{eq2.5-0411}
	\end{equation}
	where $\theta_E:X\to[\alpha,C_d]$ with $C_d$ is the doubling constant and $\alpha>0$ depends on the doubling constant, the Poincar\'e constant and the scaling constant of Poincar\' e inequalities, see \cite[Theorem 5.3 and Theorem 5.4]{Am02}
	and \cite[Therem 4.6]{AMP04}.
	
	The lower and upper approximate limits of a function $f$  on $X$ are defined respectively by 
	\begin{equation}\label{eq2.6-1511}
		f^\wedge(x):= \sup\left\{ t\in\mathbb R: \lim_{r\to 0^+}\frac{\mu(B(x,r)\cap \{y\in X: f(y)<t\}) }{\mu(B(x,r))}=0 \right\}
	\end{equation}
	and
	\begin{equation}\label{equ2.12-0912}
		f^\vee(x):= \inf\left\{ t\in\mathbb R: \lim_{r\to 0^+}\frac{\mu(B(x,r)\cap \{y\in X: f(y)>t\}) }{\mu(B(x,r))}=0 \right\}
	\end{equation}
	for $x\in X$. We denote the jump set of $f$ as 
	\[
	S_f:= \{x\in X: f^\vee(x)-f^\wedge(x)>0 \}.
	\]
	When studying the fine properties of BV functions, we consider the pointwise representatives $f^\vee$ and $f^\wedge$. The following fact clarifies the relationship between the different pointwise representatives.
	\begin{proposition}
		[Proposition 3.10 in \cite{Pa20}] Let $\Omega\subset X$ be an open set and let $f\in N^{1,1}(\Omega)$. Then $f=f^\vee=f^\wedge$ for $\mathcal H$-a.e. in $\Omega$.
		\label{prop2.2-1511}
	\end{proposition}

	By \cite[Theorem 5.3]{AMP04}, the variation measure of a ${\rm BV}$ function can be decomposed into the absolutely continuous and singular part, and the latter into the Cantor and jump parts which are all Radon measures as follows. Given an open set $\Omega\subset X$ and $f\in{\rm BV}(\Omega)$, we have for any Borel set $A\subset \Omega$,
	\begin{align}
		\|Df\|(A)=& \|Df\|^a(A)+\|Df\|^s(A)\notag\\
		=& \|Df\|^a(A)+\|Df\|^c(A)+\|Df\|^j(A)\notag\\
		=& \int_A a d\mu + \|Df\|^c(A) +\int_{A\cap S_f} \int_{f^\wedge(x)}^{f^\vee(x)}\theta_{\{f>t \}}(x)dtd\mathcal H(x) \label{eq2.6-0711}
	\end{align}
	where $a\in L^1(\Omega)$ is the density of the absolutely continuous part and the functions $\theta_{\{ f>t\}}\in[\alpha,C_d]$ as in \eqref{eq2.5-0411}. It follows that $S_f$ is $\sigma$-finite with respect to $\mathcal H$.
	
	We say that a set $A\subset X$ is $1$-quasiopen  if for every $\varepsilon>0$ there is an open set $G\subset X$ such that ${\rm Cap}_1(G)<\varepsilon$ and $A\cup G$ is open.
	We say that a set $A\subset X$ is $1$-quasiclosed if $X\setminus A$ is $1$-quasiopen.
	We refer the interested reader to \cite{Pa19,Lahti20} for more discussion on quasiopen sets.

	Given an open set $\Omega\subset X$, we say that $f$ is $1$-quasi (lower/upper semi-)continuous on $\Omega$ if for every $\varepsilon>0$ there exists an open set $G\subset X$ such that ${\rm Cap}_1(G)<\varepsilon$ and $f|_{\Omega\setminus G}$ is real-valued (lower/upper semi-)continuous. 
	
	It's a well-known fact that $N^{1,1}$-functions are quasicontinuous, see for instance \cite[Theorem 5.29]{BB15}. By Corollary 4.2 in \cite{Pa18} and Theorem 1.1 in \cite{LS17}, we obtain that ${\rm BV}$ functions have a partially analogous quasi-semicontinuity property as below.
	\begin{proposition}\label{proposition2.4-0411}
		Let $\Omega\subset X$ be open and let $f\in L^1_{\rm loc}(\Omega)$ with $\|Df\|(\Omega)<+\infty$. Then $f^\wedge$ is $1$-quasi lower semicontinuous and $f^\vee$ is $1$-quasi upper semicontinuous on $\Omega$.
	\end{proposition}

	\subsection{Chain conditions}
	\label{sec-dyadic}
	
	In this paper, we employ the following  annular chain property which is given in \cite[Section 2.5]{KN23}.
	\begin{definition} 
		\label{dyadic}
		Let $\lambda\geq1$. We say that $X$ satisfies an annular $\lambda$-chain condition at $O$  if the following holds. There are constants $c_1\geq 1, c_2\geq 1, \delta>0$ and a finite number $M<{+\infty}$  so that given $r>0$ and points $x,y\in B(O,r)\setminus B(O,r/2)$, one can find balls $B_1, B_2, \ldots, B_k$ with the following properties:
		\begin{enumerate}
			\item[1.] $k\leq M$.
			\item[2.] $B_1=B(x,r/(\lambda c_1))$, $B_k=B(y,r/(\lambda c_1))$ and the radius of  each $B_i$ is $r/(\lambda c_1)$ for $1\leq i\leq k$. 
			\item[3.] $ B_i\subset B(O,c_2r)\setminus B(O,r/c_2)$ for $1\leq i\leq k$.
			\item[4.] For each $1\leq i\leq k-1$, there is a ball $D_i\subset B_i\cap B_{i+1}$ with radius $\delta r$.
		\end{enumerate}
		If $X$ satisfies an annular $\lambda$-chain condition at $O$ for every $\lambda\geq 1$, we say that $X$ has the annular chain property.
	\end{definition}
	
	We refer to \cite{KN23, JKN23} for more discussion on the annular chain condition on complete metric measure spaces.

	\subsection{Lebesgue representative}
	Let $f\in L^1_{\rm loc}(X)$. The Lebesgue representative $f^{*}$ of $f$ is defined by
	\begin{equation}
		\label{eq-lebes}
		f^*(x):=\limsup_{r\to0}\dashint_{B(x,r)}fd\mu.
	\end{equation}
	By Lebesgue's differentiation theorem, see for instance \cite[Page 77]{HKST15}, we have that
	\begin{equation}\label{2.19-29Feb}
		f^*(x)=f(x) \text{\rm \ \ for $\mu$-a.e. $x\in X$.}
	\end{equation}
	By \cite[Theorem 1]{KKST14}, we have that if $f\in \dot {\rm BV}(X)$ then there is a constant $0<\gamma\leq \frac{1}{2}$ only depending on the doubling constant and the constants in the Poincar\'e inequality so that
	\begin{equation}
		\label{2.26-25Feb}
		(1-\gamma)f^\wedge(x) +\gamma f^\vee (x) \leq f^*(x) \leq (1-\gamma)f^\vee(x) +\gamma f^\wedge (x) \text{\rm \ \ holds true for $\mathcal H$-a.e. $x\in X$}.
	\end{equation}

	\section{Proof of Theorem \ref{theorem1.2}}

	\begin{proof}[Proof of Theorem \ref{theorem1.2}]

		Let $f\in \dot{\rm BV}(X)$.
		Here is it convenient to consider $f$ to be defined at every point in $X$.
		There is a sequence of functions $f_i\in \dot {N}^{1,1}_{\rm loc}(X)$  so that $f_i\to f$ in $L^{1}_{\rm loc}(X)$ and 
		\[ 
		\liminf_{i\to+\infty}\int_X g_{f_i}d\mu<+\infty.
		\]
		Passing to a subsequence (not relabeled), we also have $f_i(x)\to f(x)$ for $\mu$-a.e. $x\in X$.
		We have
		\begin{equation}\label{eq2.21-0912}
			\text{\rm $f^\vee(\gamma(t))=f^\wedge(\gamma(t))=f(\gamma(t))$ for all $t\notin N_\gamma$}
		\end{equation}
		for a subset $N_{\gamma}\subset[0,+\infty)$ with $\mathcal L^1(N_\gamma)=0$;
		this is obtained because $\mu(S_f)=0$ and so  $\mathcal H^1(\gamma^{-1}(S_f))=0$ for $1$-a.e. infinite curve $\gamma$, and hence for ${\rm AM}$-a.e. infinite curve $\gamma$ parameterized by arc-length, one can pick the set $N_\gamma\subset[0,+\infty)$ with $\mathcal L^1(N_\gamma)=0$ so that $\gamma^{-1}(S_f)\subset N_\gamma$.
		Moreover, one can pick $N_\gamma$ such that 
		for ${\rm AM}$-a.e. infinite curve $\gamma$ parameterized by arc-length,
		$\lim_{i\to+\infty} f_{i}(\gamma(t))=f(\gamma(t))$ for all $t\notin N_\gamma$.
		By Lemma \ref{lem2.2-2508},  we have that
		\[
		\liminf_{i\to+\infty}\int_0^{+\infty} g_{f_{i}}\circ \gamma(t)dt  
		=  \liminf_{i\to+\infty}\int_\gamma g_{f_{i}}ds<+\infty
		\]
		for {\rm AM}-a.e. infinite curve $\gamma$ parameterized by arc-length.
		
		Thus for ${\rm AM}$-a.e. infinite curve $\gamma$ parameterized by arc-length,
		for some subsequence $f_{i,\gamma}$ we have that
		$f_{i,\gamma}\in L^1_{\rm loc}(X)$ such that 
		\begin{equation}\label{eq2.20-0912}
			\sup_{i\in\mathbb N} \int_0^{+\infty}(g_{f_{i,\gamma}}{\circ\gamma})(t)dt<+\infty{\rm \ \ and \ \ }   
			\lim_{i\to+\infty} f_{i,\gamma}(\gamma(t))=f(\gamma(t)) \text{\rm \ \ for all $t\notin N_\gamma$}.
		\end{equation}
		
		By the definition of upper gradients, we have that  for ${\rm AM}$-a.e. infinite curve $\gamma$ parameterized by arc-length and for $0<t_1<t_2<+\infty$, 
		\begin{equation}\label{eq2.22-0912}
			|  f_{i,\gamma}(\gamma(t_1))-f_{i,\gamma}(\gamma(t_2)) | \leq \int_{t_1}^{t_2} (g_{f_{i,\gamma}}\circ \gamma) (t)dt=: \nu_{i,\gamma}([t_1,t_2])
		\end{equation}
		for each $i\in\mathbb N$ where $\nu_{i,\gamma}$ is a measure on $[0,+\infty)$ defined by $d\nu_{i,\gamma}(t):=(g_{f_{i,\gamma}}\circ\gamma)(t)dt$ for  $i\in\mathbb N$.
		
		Let $\gamma$ be an infinite curve so that $N_\gamma, \{ f_{i,\gamma}\}_{i\in\mathbb N}$ and $\{\nu_{i,\gamma}\}_{i\in\mathbb N}$ defined as above satisfy 
		\eqref{eq2.21-0912}-\eqref{eq2.20-0912}-\eqref{eq2.22-0912}. It then suffices to prove that the limit
		\[
		\lim_{t\to+\infty,t\notin N_\gamma}f(\gamma(t))
		\]
		exists. Since the sequence $\nu_{i,\gamma}$ is bounded on $[0,+\infty)$
		by \eqref{eq2.20-0912}, it follows from \cite[Theorem 1.41]{EG15}  that there is a subsequence $\{\nu_{i_k,\gamma}\}_{k\in\mathbb N}$ so that it weakly* converges to a measure $\nu$ on $[0,\infty)$ and  we then have from the lower semicontinuity properties in  \cite[Theorem 1.40]{EG15}  that 
		\begin{equation}\label{eq2.23-0912}
			\nu([0,+\infty)) \leq \liminf_{k\to+\infty} \nu_{i_k,\gamma}([0,+\infty))\leq  \sup_{i\in\mathbb N} \int_0^{+\infty}(g_{f_{i,\gamma}}{\circ\gamma})(t)dt =:m_\gamma<+\infty,
		\end{equation}
		where $m_\gamma$ is finite by \eqref{eq2.20-0912},
		and that
		\begin{equation}\label{eq2.24-0912}
			\limsup_{k\to+\infty}  \nu_{i_k,\gamma} ([a,b]) \leq \nu ([a,b]) \text{\rm \ \ for all closed intervals $[a,b]\subset(0,+\infty)$}.
		\end{equation}
		Let $t_1,t_2\in (0,+\infty)$ be arbitrary so that $t_2> t_1$ and $t_1,t_2\notin N_\gamma$. We also denote by $\{f_{i_k,\gamma}\}_{k\in\mathbb N}$ the subsequence of $\{f_{i,\gamma}\}_{i\in\mathbb N}$ corresponding to the subsequence $\{\nu_{i_k,\gamma}\}_{k\in\mathbb N}$.
		We have that
		\begin{align*}
			|f (\gamma(t_1))-f (\gamma(t_2))|
			&=\lim_{k\to+\infty} |f_{i_k,\gamma}(\gamma(t_1))-f_{i_k,\gamma}(\gamma(t_2))| \\
			&\leq  \limsup_{k\to+\infty}  \int_{t_1}^{t_2} (g_{f_{i_k,\gamma}}\circ \gamma)(t) dt\\
			&=  \limsup_{k\to+\infty} \nu_{i_k,\gamma}([t_1,t_2])\\
			&\leq \nu ([t_1,t_2])
		\end{align*}
		where the first equation is given by \eqref{eq2.20-0912}, the first inequality is obtained by \eqref{eq2.22-0912} and the last estimate is given by \eqref{eq2.24-0912}. Since $\nu$ is a bounded measure on $[0,+\infty)$ by \eqref{eq2.23-0912}, letting $t_1,t_2\to+\infty$, it follows that there is $c_\gamma\in\mathbb R$ such that
		\begin{equation}\notag
			\lim_{t\to+\infty, t\notin N_\gamma}f(\gamma(t))=c_\gamma.
		\end{equation}
		The claim follows.
	\end{proof}
	
	\section{Proofs of Theorem \ref{theorem1.5} and Corollary \ref{thm1.1}}
	\subsection{Proof of Theorem \ref{theorem1.5}}
	\ 
	
	Recall that for $f\in \dot N^{1,1}(X)$, $f$ is absolutely continuous on $1$-a.e. infinite curve $\gamma$ and ${\rm Mod}_1(\{\gamma\in\Gamma^{+\infty}: \int_\gamma g_fds=+\infty \})=0$. Then we obtain that
	\begin{equation}\label{thm1.1-0711}
		\text{\rm for every $f\in \dot {N}^{1,1}(X)$, the limit\ } \lim_{t\to+\infty} f (\gamma(t)) \text{\rm \ \ exists for $1$-a.e. infinite curve $\gamma$.}
	\end{equation}
	By \cite[Proposition 5.2]{La20}, we have the following lemma saying
	that any bounded variation function with ``small jumps" can be approximated by Sobolev functions.
	
	\begin{lemma}\label{lem3.1-0711}
		
		Suppose that $(X,d,\mu)$ is a complete doubling unbounded metric measure space supporting a $1$-Poincar\'e inequality.
		Let $\varepsilon>0$. Let  $f\in \dot{\rm BV}(X)$ and
		\begin{equation}\label{eq3.1-2711}
			X_{\varepsilon}:= \{x\in X: f^\vee(x)-f^\wedge(x) <\varepsilon  \}.
		\end{equation}
		Then there exists  a function  $h_{\varepsilon}\in \dot {N}^{1,1}(X_\varepsilon)$ so that
		\[
		f^\vee(x)-10\varepsilon \leq h_{\varepsilon}(x)\leq f^\wedge(x)+ 10\varepsilon \ \ \text{\rm for all  $x\in X_\varepsilon$.}
		\]
		
		In particular,
		\[
		\max \{|h_\varepsilon(x)-f^\vee(x)|, |h_\varepsilon(x)-f^\wedge(x)| \} \leq 10 \varepsilon \ \  \text{\rm for all  $x\in X_\varepsilon$.}
		\]
	\end{lemma}
	\begin{proof}\ 
		
		By Proposition \ref{proposition2.4-0411}, we have that $X_\varepsilon$ is $1$-quasiopen and for each $i\in\mathbb N$, the set $
		A_i:=\{ x\in X_\varepsilon: f^\vee(x)\geq (i+1)\varepsilon \}$
		is $1$-quasiclosed and the set
		$
		D_i:=\{x\in X_\varepsilon: f^\wedge(x)>(i-2)\varepsilon  \}
		$
		is $1$-quasiopen. 
		By the  arguments of \cite[Proposition 5.2]{La20} applied to $X_\varepsilon, A_i, D_i$, we obtain the claim.
	\end{proof}
	

	\begin{lemma}\label{lem3.1-0211}
		Let $\varepsilon> 0$. Suppose that  $f\in \dot{\rm BV}(X)$. We denote 
		\[
		S_{f,\varepsilon}:=\{x\in X: f^\vee(x)-f^\wedge(x) \geq \varepsilon  \}.
		\]
		Then 
		\[
		{\rm Mod}_1(\Gamma_\varepsilon)=0,
		\]
		where $\Gamma_\varepsilon$ consists all infinite curves $\gamma\in\Gamma^{+\infty}$ satisfying
		\begin{equation}\label{eq4.2-0912}
			\# \left\{i\in\mathbb N: A_i \cap \gamma\cap S_{f,\varepsilon}\neq\emptyset \right\}=+\infty,
		\end{equation}
		where $A_i$ is the annulus $B(O,2^{i+1})\setminus B(O,2^i)$.
	\end{lemma}
	
	Roughly speaking, each $\gamma$ in $\Gamma_\varepsilon$ meets the jump set $S_{f,\varepsilon}$ at infinity.
	
	\begin{proof}\ 
		By \eqref{eq2.6-0711}, we have that 
		\[
		\int_{S_f} f^\vee(x)-f^\wedge(x) d\mathcal H(x)\leq \frac{1}{\alpha} \|Df\|^j (X) <+\infty
		\]
		where $\alpha>0$ as defined in $\eqref{eq2.5-0411}$ and \eqref{eq2.6-0711}.  It follows that  for every $\varepsilon>0$, 
		\begin{equation}\label{equ4.2-0912}
			\mathcal H (S_{f,\varepsilon})<+\infty.
		\end{equation}

		For $i\in\mathbb N$, we let $\Gamma_i$ be the family of all infinite curves intersecting
		$A_i\cap S_{f,\varepsilon}$. 
		One has that for all $j\in\mathbb N$,
		\begin{equation} \notag
			\Gamma_\varepsilon\subseteq \bigcup_{i=j}^{+\infty} \Gamma_i
		\end{equation}
		and hence
		\begin{equation}\label{eq3.1-0711}
			{\rm Mod}_1(\Gamma_\varepsilon) \leq \lim_{j\to+\infty}\sum_{i=j}^{+\infty} {\rm Mod}_1(\Gamma_i).
		\end{equation}
		By \eqref{eq:Mod and H}, we have that there is a constant $C>0$ so that for each $i$, 
		\[
		{\rm Mod}_1(\Gamma_i)\leq C \mathcal H(A_i\cap S_{f,\varepsilon}).
		\]
		Summing over $i\in\mathbb N$ we obtain from \eqref{equ4.2-0912}  that
		\[
		\lim_{j\to+\infty}\sum_{i=j}^{+\infty} \mathcal H(A_i\cap S_{f,\varepsilon}) =0.
		\]
		By 
		\eqref{eq3.1-0711} it follows that ${\rm Mod}_1(\Gamma_\varepsilon)=0$ which is the claim.
	\end{proof}

	\begin{proof}[Proof of Theorem \ref{theorem1.5}]
		\ 
		
		Just as in \eqref{equ4.2-0912}, for every $\varepsilon>0$ we have 
		\[
		\mathcal H(S_{f,\varepsilon}) <+\infty
		\]
		where $S_{f,\varepsilon}:= \{x\in X: f^\vee(x)-f^\wedge(x)\geq \varepsilon\}$. By Lemma \ref{lem3.1-0211}, 
		\begin{equation}\label{equ3.4-0711}
			{\rm Mod}_1(\Gamma_{\varepsilon})=0
		\end{equation}
		where $\Gamma_{\varepsilon}$ is the collection of all infinite curves $\gamma\in\Gamma^{+\infty}$ satisfying 
		\[
		\# \left\{i\in\mathbb N: A_i \cap \gamma\cap S_{f,\varepsilon}\neq\emptyset \right\}=+\infty
		\]
		where $A_i$ is the annulus $B(O,2^{i+1})\setminus B(O,2^i)$.

		Let $\varepsilon_i>0$ be a sequence so that $\varepsilon_i>\varepsilon_{i+1}$ and $\varepsilon_i\to 0$ as $i\to+\infty$. By \eqref{equ3.4-0711} and the subadditivity of ${\rm Mod}_1$, we have 
		\begin{equation}\label{eq4.7-1stFeb}
			{\rm Mod}_1(\cup_{i\in\mathbb N}\Gamma_{\varepsilon_i})=0.
		\end{equation}
		It then suffices to prove that 
		\begin{equation}\label{eq3.3-0711}
			\lim_{t\to+\infty} f^\vee(\gamma(t)) {\rm \ and\ } \lim_{t\to+\infty} f^\wedge(\gamma(t)) \text{\rm \ \ exist for $1$-a.e. infinite curve $\gamma\in \Gamma^{+\infty}\setminus (\cup_{i\in\mathbb N}\Gamma_{\varepsilon_i})$}=:\Gamma.
		\end{equation}
		Here if the claim \eqref{eq3.3-0711} holds then the value of both limits is the same by \eqref{eq4.7-0912}.
		For each $i\in\mathbb N$, we have from the definition of $\Gamma_{\varepsilon_i}$ as above that

		\begin{equation}\label{eq4.7-0912}
			f^\vee(\gamma(t))-f^\wedge(\gamma(t))<\varepsilon_i \text{\rm \ \ for all   $\gamma\notin \Gamma_{\varepsilon_i}$ and for all $t$ sufficiently large} .
		\end{equation}
		For each $i\in\mathbb N$, we have from Lemma \ref{lem3.1-0711} that there exists $h_{\varepsilon_i}\in \dot N^{1,1}(X_{\varepsilon_i})$ so that  
		\begin{equation}\label{eq3.4-0711}
			\max \{|h_{\varepsilon_i}(x)-f^\vee(x)|, |h_{\varepsilon_i}(x)-f^\wedge(x)| \} \leq 10 \varepsilon_i
		\end{equation}
		for all  $x\in X_{\varepsilon_i}$ where
		\[
		X_{\varepsilon_i}= X\setminus S_{f,\varepsilon_i}=\{x\in X: f^\vee(x)-f^\wedge(x)<\varepsilon_i \}.
		\] 
		By the same argument as in \eqref{thm1.1-0711},
		there is a family $\Gamma_i\subset\Gamma$ with ${\rm Mod}_1(\Gamma_i)=0$ so that 
		\[
		\lim_{t\to+\infty}h_{\varepsilon_i}(\gamma(t))
		\]
		exists for all $\gamma\in \Gamma\setminus \Gamma_i$. Then ${\rm Mod}_1(\cup_{i\in\mathbb N}\Gamma_i)=0$ and 
		\begin{equation}\label{equ3.5-0711}
			\lim_{t\to+\infty}h_{\varepsilon_i}(\gamma(t)) \text{\rm \ \ exists for all $i\in\mathbb N$, for all $\gamma\in \Gamma\setminus (\cup_{i\in\mathbb N}\Gamma_i)$}.
		\end{equation}
		Let $\Gamma_{\rm good}:= \Gamma\setminus (\cup_{i\in\mathbb N}\Gamma_i)$. To prove \eqref{eq3.3-0711}, we only need to show that 
		\begin{equation}
			\label{eq3.5-0711} \lim_{t\to+\infty}f^\vee(\gamma(t)) {\rm\ and\ }\lim_{t\to+\infty}f^\wedge(\gamma(t)) \text{\rm \ \ exist for all $\gamma\in\Gamma_{\rm good}$}.
		\end{equation}
		We prove \eqref{eq3.5-0711} by a contradiction. Let $\gamma_{\rm good}\in \Gamma_{\rm good}$ be arbitrary. Suppose that 
		\[
		\lim_{t\to+\infty}f^\vee(\gamma_{\rm good}(t))  \text{\rm \ \   does not exist }
		\]
		and hence there exist $\varepsilon_{\gamma_{\rm good}}>0$ and two sequences $\{t_{j}\}_{j\in\mathbb N}$, $\{ s_{j}\}_{j\in\mathbb N}$ so that 
		\[
		\lim_{j\to+\infty}t_j=\lim_{j\to+\infty}s_j=+\infty
		\]
		and
		\[
		f^\vee(\gamma_{\rm good}(t_j)) - f^\vee(\gamma_{\rm good}(s_j))\geq \varepsilon_{\gamma_{\rm good}} \text{\rm \ \ for all $j\in\mathbb N$}.
		\]
		Applying \eqref{eq3.4-0711}, we have that for all $i\in\mathbb N$,
		\[
		|h_{\varepsilon_{i}} (\gamma_{\rm good}(t))- f^\vee(\gamma_{\rm good}(t))| \leq 10 \varepsilon_{i} \text{\rm \ \ holds for all $t$ sufficiently large.}
		\]
		Combining two estimates above, we obtain that for all $i\in\mathbb N$,
		\[
		|h_{\varepsilon_i} (\gamma_{\rm good}(t_j))- h_{\varepsilon_i}(\gamma_{\rm good}(s_j))| \geq \varepsilon_{\gamma_{\rm good}}- 20\varepsilon_i
		\]
		holds for all  $j$ sufficiently large. Since $\lim_{i\to+\infty}\varepsilon_i=0$, we pick $i$ sufficiently large so that the right hand side above is bounded from below  by $\varepsilon_{\gamma_{\rm good}}/2$. Then the limit at infinity  of  $h_{\varepsilon_i}$, for such $i$ large, along such $\gamma_{\rm good}\in\Gamma_{\rm good}$ does not exist which is a contradiction with \eqref{equ3.5-0711}. 
		Thus $\lim_{t\to+\infty}f^\vee(\gamma_{\rm good}(t))$ exists. Similarly, we also obtain that $\lim_{t\to+\infty}f^\wedge(\gamma_{\rm good}(t))$ exists.
		We conclude that $\eqref{eq3.5-0711}$ holds and so the  claim follows.
	\end{proof}

	\subsection{Proof of Corollary \ref{thm1.1}}
	\label{sec3}
	\
	

	%
	
	\begin{theorem}\label{thm3.5-0911}
		Suppose that $(X,d,\mu)$ is a complete doubling unbounded metric measure space supporting a $1$-Poincar\'e inequality
		and satisfying an annular $\lambda$--chain condition, where $\lambda\ge 1$ is the scaling factor in the Poincar\'e inequality.
		Then for every $f\in \dot{\rm BV}(X)$, there exists a finite constant $c_f\in\mathbb R$ so that
		\begin{equation} \label{4.13-29Feb}\lim_{t\to+\infty}f^\vee(\gamma(t))=\lim_{t\to+\infty}f^\wedge(\gamma(t))=\lim_{t\to+\infty}f^*(\gamma(t))=c_f
		\end{equation}
		for ${\rm 1}$-a.e. infinite curve $\gamma$.
	\end{theorem}
	\begin{proof}\ 
		
		By Theorem \ref{theorem1.5} and Remark \ref{rem1.5-12nd4}, the three limits at infinity exist and are equal to the same value for $1$-a.e. infinite curve, i.e.  for every $f\in\dot{\rm BV}(X)$, for $1$-a.e. $\gamma\in\Gamma^{+\infty}$, there is $c_{f,\gamma}\in\mathbb R$ so that 
		\begin{equation}\label{4.14-29Feb}
			\text{\rm $\lim_{t\to+\infty}f^\vee(\gamma(t))=\lim_{t\to+\infty}f^\wedge(\gamma(t))=\lim_{t\to+\infty}f^*(\gamma(t))=c_{f,\gamma}$}.
		\end{equation}
		Notice that $f^*\in\dot{\rm BV}(X)$ since $f\in\dot{\rm BV}(X)$ and since $f^*(x)=f(x)$ for $\mu$-a.e. $x\in X$ by \eqref{2.19-29Feb}.
		Applying the Poincar\'e inequality \eqref{eq2.5-1611} with respect to the ${\rm BV}$-case, we use the same arguments as in the case of Sobolev functions (see \cite[Proof of Theorem 1.1]{KN23})  to obtain the uniqueness part, i.e. there is a finite constant $c_f\in \mathbb R$ so that 
		\[
		\lim_{t\to+\infty}f^*(\gamma(t))=c_f {\rm \ \ for \ } \text{\rm \ $1$-a.e.\ } \gamma\in\Gamma^{+\infty}.
		\]
		Combining this with \eqref{4.14-29Feb}, the uniqueness part for $f^\vee$ and $f^\wedge$ is obtained
		and so
		the claim \eqref{4.13-29Feb} holds. The proof completes.
	\end{proof}

	\begin{proof}
		[Proof of Corollary \ref{thm1.1}]\

		By  Theorem \ref{thm3.5-0911}, we obtain that the existence and uniqueness of limits at infinity of $\dot{\rm BV}$-functions along $1$-a.e. infinite curve hold true.
		
		Next, we consider the case of $f\in {\rm BV}(X)$. Since $f\in \dot{\rm BV}(X)$, there is a finite constant $c_f$ so that $\lim_{t\to+\infty}f^*(\gamma(t))=c_f$ for $1$-a.e. infinite curve $\gamma\in\Gamma^{+\infty}$. We will show that $c_f=0$. We argue by contradiction.  Suppose that $|c_f|>0$.
		Consider the family $\Gamma$ of all infinite curves $\gamma\in\Gamma^{+\infty}$ for which
		$\lim_{t\to+\infty}f^*(\gamma(t))=c_f$; we can assume that ${\rm Mod}_1(\Gamma^{+\infty})>0$
		and
		${\rm Mod}_1(\Gamma)>0$.

		For each $\gamma\in\Gamma$, there is $s_\gamma\ge 0$ 
		so that $|f^*(\gamma(s))| \ge |c_f|/2$ for all $s\ge s_\gamma$. 
		Let $M>0$. For every $\gamma \in\Gamma$, we have
		\[
		\int_{\gamma}M^{-1}|f^*|\,ds
		=\int_{0}^{+\infty}M^{-1}|f^*(\gamma(s))|\,ds
		\ge\int_{s_\gamma}^{+\infty}M^{-1}|f^*(\gamma(s))|\,ds
		\ge \int_{s_\gamma}^{+\infty}M^{-1}|c_f|/2\,ds
		=+\infty.
		\]
		Since now $f\in L^1(X)$, we have
		\[
		{\rm Mod}_1(\Gamma)\le M^{-1}\int_X |f^*|\,d\mu \to 0
		\]
		as $M\to+\infty$. This is a contradiction, and so the proof completes.
	\end{proof}
	
	\section{Proof of Theorem \ref{thm1.2}}
	Recall that $\mu(X)>0$ by our standing assumption, see the beginning of Section \ref{section2}. We will use it in the proof of the following lemma.
	\begin{lemma}
		\label{lem4.1-1809} 
		Let $(X,d,\mu)$ be a complete doubling unbounded metric measure space supporting a $1$-Poincar\'e inequality.
		Suppose that $\mathcal R_1=+\infty$. Then the following statements hold true:
		\begin{enumerate}
			\item 
			There is a function $f\in {\rm BV}(X)\setminus \dot{N}^{1,1}(X)$ so that 
			\begin{equation}\label{eq4.1-1411}
				\lim_{t\to+\infty} f^\vee(\gamma(t)) {\rm \ \ and\ \ }  \lim_{t\to+\infty} f^\wedge(\gamma(t)) \text{\rm \ \ fail to exist for all infinite curves $\gamma$}.
			\end{equation}
			\item   There is a function $f\in N^{1,1}(X)$ so that 
			\begin{equation}\label{eq4.2-1411}
				\lim_{t\to+\infty} f(\gamma(t)) \text{\rm \ \ fails to exist for all infinite curves $\gamma$}.
			\end{equation}
			\item  
			There is a function $f\in \dot N^{1,1}(X)\setminus N^{1,1}(X)$ so that 
			\begin{equation}\label{eq5.2-revise}
				\lim_{t\to+\infty} f(\gamma(t)) \text{\rm \ \ fails to exist for all infinite curves $\gamma$}.
			\end{equation}
			
		\end{enumerate}
	\end{lemma}
	\begin{proof}\ 
		
		Firstly, we note that $X$ is \emph{quasiconvex}, meaning that for every
		pair of points $x,y\in X$ there is a curve $\gamma:[0,\ell_{\gamma}]\to X$ with $\gamma(0)=x$,
		$\gamma(\ell_{\gamma})=y$, and $\ell_{\gamma}\le Cd(x,y)$,
		where $C$ is a constant that
		only depends on the doubling constant of the measure and the constants
		in the Poincar\'e inequality, see e.g. \cite[Theorem 4.32]{BB15}.
		Thus a biLipschitz change in the metric gives a geodesic space
		(see \cite[Section 4.7]{BB15}).
		Since the theorem
		is easily seen to be invariant under such a biLipschitz
		change in the metric, we can assume that $X$ is geodesic.
		Then from \cite{Bu99} we know that every sphere
		$\overline{B}(x,r)\setminus B(x,r)$ has zero $\mu$-measure.
		
		By the assumption that 
		\[
		\mathcal R_1=\sup_{j\in\mathbb N} \frac{2^j}{\mu(B(O,2^{j+1})\setminus(B(O,2^j))}=+\infty,
		\]
		we have
		\[
		\inf_{i\in\mathbb N}  \frac{\mu(B(O,2^i))}{2^i}=0.
		\]
		Using also doubling, there is a sequence $\{i_j\}_{j\in\mathbb N}$ so that for each   $j\in\mathbb N$,
		\begin{equation}\label{5.3-14thMarch}
			\frac{\mu(B(O,2^{i_j+2}))}{2^{i_j}} \leq \frac{1}{2^j} {\rm \ \  and \ so \ \ } \sum_{j\in\mathbb N}    \frac{\mu(B(O,2^{i_j+2}))}{2^{i_j}} \leq \sum_{j\in\mathbb N}\frac{1}{2^j}<+\infty.
		\end{equation}
		
		
		We will denote by $\overline{A}(r,\varepsilon)$ the annulus
		$\overline{B}(O,r+\varepsilon)\setminus B(O,r)$ for $r>0, \varepsilon>0$. 
		Then for each $j\in\mathbb N$, we find an annulus $\overline{A}(a_{j,0},1)=\overline{B}(O,a_{j,0} + 1)\setminus B(O,a_{j,0})$
		with $a_{j,0}\in \{2^{i_j},\ldots,2^{i_j+1}-1\}$, so that  
		\begin{equation}\label{5.4-21March}
			{\mu(\overline{A}(a_{j,0},1) )}\leq \frac{\mu(\overline{A}(2^{i_j},2^{i_j}))}{2^{i_j}}.
		\end{equation}
		We will construct  a function $f\in {\rm BV}(X)\setminus \dot N^{1,1}(X)$ with $0\leq f\leq 1$ and its support in $\cup_{j\in\mathbb N}\overline{A}(a_{j,0},1)$ satisfying the claim \eqref{eq4.1-1411}. 
		Combining \eqref{5.3-14thMarch} and \eqref{5.4-21March}, we get $\mu( \cup_{j\in\mathbb N}\overline{A}(a_{j,0},1))<\infty$ and so $f\in L^1(X)$ for every function $0\leq f\leq 1$ with its support in $\cup_{j\in\mathbb N}\overline{A}(a_{j,0},1)$. Hence we only need to construct a function $f$ satisfying \eqref{eq4.1-1411} so that 
		\begin{equation}\label{5.8-15thMarch}
			0\leq f\leq 1,\, {\rm spt}(f)\subseteq \cup_{j\in\mathbb N}\overline{A}(a_{j,0},1){\rm \ \ and\ \ }\| Df\|(\overline{A}(a_{j,0},1))\leq \frac{6}{2^j} {\rm \ \ for \ each \ } j\in\mathbb N
		\end{equation}
		where ${\rm spt}(f)$ is the support of $f$.
		Now, we fix $j\in \mathbb N$ so that \eqref{5.4-21March} holds. Let $b_{j,0}:=a_{j,0}+1/3$ and $ c_{j,0}=b_{j,0}+1/3$.
		Then \eqref{5.3-14thMarch}-\eqref{5.4-21March} gives that
		\[
		\mu(\overline{A}(a_{j,0},1/3))\leq   \mu(\overline{A}(a_{j,0},1))\leq \frac{1}{2^j}
		{\rm\ \ and\ \ }   \frac{ \mu(\overline{A}(a_{j,0},1/3))}{1/3}\leq \frac{ \mu(\overline{A}(a_{j,0},1))}{1/3}\leq \frac{3}{2^j} 
		\]
		and similarly
		\[
		\mu(\overline{A}(c_{j,0},1/3))\leq   \mu(\overline{A}(a_{j,0},1))\leq \frac{1}{2^j}
		{\rm\ \ and\ \ }   \frac{ \mu(\overline{A}(c_{j,0},1/3))}{1/3}\leq \frac{ \mu(\overline{A}(a_{j,0},1))}{1/3}\leq \frac{3}{2^j}. 
		\]
		Considering two annuli $\overline{A}(a_{j,0},1/3)$ and  $\overline{A}(c_{j,0},1/3)$, we divide each of them into two smaller annuli $\overline{A}(a_{j,0},2^{-1}/3), \overline{A}(a_{j,0}+2^{-1}/3,2^{-1}/3)$,    $\overline{A}(c_{j,0},2^{-1}/3), \overline{A}(c_{j,0}+2^{-1}/3,2^{-1}/3)$. Then there are at least two annuli, denoted by $\overline{A}(a_{j,1},2^{-1}/3)$ and $\overline{A}(c_{j,1},2^{-1}/3)$, so that
		\[
		\overline{A}(a_{j,1},2^{-1}/3)\subset \overline{A}(a_{j,0},1/3){\rm \ \ and\ \ }
		\frac{\mu(\overline{A}(a_{j,1},2^{-1}/3))}{2^{-1}/3}\leq \frac{ \mu(\overline{A}(a_{j,0},1/3))}{1/3}\leq\frac{3}{2^j}
		\]
		and 
		\[
		\overline{A}(c_{j,1},2^{-1}/3)\subset \overline{A}(c_{j,0},1/3){\rm \ \ and\ \ }
		\frac{\mu(\overline{A}(c_{j,1},2^{-1}/3))}{2^{-1}/3}\leq \frac{ \mu(\overline{A}(c_{j,0},1/3))}{1/3}\leq\frac{3}{2^j}
		\]
		here $a_{j,1}\in \{a_{j,0}, a_{j,0}+2^{-1}/3\}$ and $c_{j,1}\in \{c_{j,0}, c_{j,0}+2^{-1}/3\}$. Similarly dividing $\overline{A}(a_{j,1},2^{-1}/3)$ and $\overline{A}(c_{j,1},2^{-1}/3)$ each into two smaller annuli, we obtain by induction that there are two sequences $\{ \overline{A}(a_{j,k},2^{-k}/3)\}_{k\in\mathbb N}$ and $\{ \overline{A}(c_{j,k},2^{-k}/3)\}_{k\in\mathbb N}$ so that for each $k$,
		\[
		\begin{cases}
			\overline{A}(a_{j,k+1},2^{-(k+1)}/3)\subset  \overline{A}(a_{j,k},2^{-k}/3)\subset \overline{A}(a_{j,0},1/3),\\
			\frac{\mu(\overline{A}(a_{j,k+1},2^{-(k+1)}/3))}{2^{-(k+1)}/3}\leq \frac{\mu(\overline{A}(a_{j,k},2^{-k}/3))}{2^{-k}/3}\leq \frac{ \mu(\overline{A}(a_{j,0},1/3))}{1/3}\leq\frac{3}{2^j}
		\end{cases}
		\]
		and 
		\[
		\begin{cases}
			\overline{A}(c_{j,k+1},2^{-(k+1)}/3)\subset  \overline{A}(c_{j,k},2^{-k}/3)\subset \overline{A}(c_{j,0},1/3),\\
			\frac{\mu(\overline{A}(c_{j,k+1},2^{-(k+1)}/3))}{2^{-(k+1)}/3}\leq \frac{\mu(\overline{A}(c_{j,k},2^{-k}/3))}{2^{-k}/3}\leq \frac{ \mu(\overline{A}(c_{j,0},1/3))}{1/3}\leq\frac{3}{2^j}.
		\end{cases}
		\]
		Moreover there are $a_j<c_j$ so that $\lim_{k\to+\infty}a_{j,k}=a_j$ and $\lim_{k\to+\infty}c_{j,k}=c_j$.
		The two above claims give that there are two sequences $\{u_{j,k}\}_{k\in\mathbb N}$ and $\{v_{j,k}\}_{k\in\mathbb N}$ of Lipschitz functions so that 
		\begin{equation}\label{5.11-15thMarch}
			u_{j,k}=1 {\rm \ on\ }B(O,a_{j,k}),\ u_{j,k}=0 {\rm \ on\ }X\setminus B(O,a_{j,k}+2^{-k}/3), \ 
			\int_X{g_{u_{j,k}}}d\mu\leq \frac{3}{2^j}
		\end{equation}
		and
		\begin{equation}\label{5.12-15thMarch}
			v_{j,k}=1 {\rm \ on\ }B(O,c_{j,k}),\ v_{j,k}=0 {\rm \ on\ }X\setminus B(O,c_{j,k}+2^{-k}/3),\ 
			\int_X{g_{v_{j,k}}}d\mu\leq \frac{3}{2^j}.
		\end{equation}
		Let 
		\begin{equation}\label{eq5.12-15March}
			h_{j,k}=(1-u_{j,k})\chi_{\overline{A}(a_{j,0},1/3)}+\chi_{\overline{A}(b_{j,0},1/3)}+v_{j,k}\chi_{\overline{A}(c_{j,0},1/3)}.
		\end{equation}
		It follows from \eqref{5.11-15thMarch}-\eqref{5.12-15thMarch} that the sequence $h_{j,k}$ converges
		in $L^1(X)$ to some function $h_j\in \dot{\rm BV}(X)$ as $k\to+\infty$ so that 
		\begin{equation}\label{5.13-15thMarch}
			\text{\rm $h_j=1$ in $B(O,c_j)\setminus \overline{B}(O,a_j)$, \  $h_j$ has jumps on $\{x\in X: d(O,x)=a_j\}\cup\{x\in X: d(O,x)=c_j\}$},
		\end{equation}
		and 
		\begin{equation}\label{5.14-15thMarch}
			\text{\rm ${\rm spt}(h_j)\subseteq \overline{B}(O,c_j)\setminus B(O,a_j)\subseteq \overline{A}(a_{j,0},1)$ \ \ and\ \ $ \|Dh_j\|(\overline{A}(a_{j,0},1))\le \frac{6}{2^j}.$}
		\end{equation}
		Let $f=\sum_{j\in\mathbb N}h_j\chi_{\overline{A}(a_{j,0},1)}$. By \eqref{5.13-15thMarch}-\eqref{5.14-15thMarch}, we obtain  that
		\[
		\lim_{t\to+\infty}f^\vee(\gamma(t)), \lim_{t\to+\infty}f^\wedge(\gamma(t)) \text{\rm\ \ fail to exist for all infinite curves $\gamma$}
		\]
		and that
		$0\leq f\leq 1, {\rm spt}(f)\subseteq \cup_{j\in\mathbb N}\overline{A}(a_{j,0},1)$ and $\|Df\|(\overline{A}(a_{j,0},1))\le \frac{6}{2^j}$ for each $j\in\mathbb N$ which is \eqref{5.8-15thMarch}, and so the first  claim follows.
		
		For the second  claim, we let $f=\sum_{j\in\mathbb N}h_{j,k}\chi_{\overline{A}(a_{j,0},1)}$ for a fixed $k$ and where $h_{j,k}$ are defined by \eqref{eq5.12-15March}. One can check from \eqref{5.11-15thMarch}-\eqref{5.12-15thMarch}-\eqref{eq5.12-15March} that this function $f\in N^{1,1}(X)$ satisfies \eqref{eq4.2-1411}. 
		
		Finally,
		by letting
		\[
		f(x)=\sum_{j\in\mathbb N}(1-u_{j,k}){\rm \ \ for \ a \ fixed\ }k,
		\]
		we obtain from \eqref{5.11-15thMarch} that $g:=\sum_{j\in\mathbb N} g_{u_{j,k}}$ is an integrable
		upper gradient of $f$ on $X$ and $f\to +\infty$ as $d(O,x)\to+\infty$. 
		Notice that $X$ is connected since it supports a $1$-Poincar\'e inequality, see \cite[Proposition 4.2]{BB15}, and hence by \cite[Corollary 3.8]{BB15} we obtain a relative upper volume decay, i.e. there are constants $C>0, s>0$ such that for all $0<r<R<\tfrac 12 \textrm{diam} X, x\in X$, 
		\[
		\frac{\mu(B(x,r))}{\mu(B(x,R))}\leq C \left(\frac{r}{R} \right)^s.
		\]
		In particular, in an unbounded space, we obtain that $\mu(X)=+\infty$ because 
		\[
		\mu(X)\geq \mu(B(O,R)) \geq \frac{1}{C} \frac{\mu(B(O,r))}{r^s} R^s  \to +\infty  {\rm \ \ as\ \ }R\to+\infty
		\]
		for some fixed $r>0$. 
		Combining this with $f(x)\to+\infty$ as $d(O,x)\to+\infty$, we see that the function $f$ is not integrable.
		Therefore, \eqref{eq5.2-revise}
		holds true and so the third claim follows.
		
		The proof completes.
	\end{proof}
	\begin{proof}
		[Proof of Theorem \ref{thm1.2}]

		\ 
		
		$(1)\Longrightarrow (4)$: By \cite{KN23}, we obtain from $f\in\dot N^{1,1}(X)$ that $\lim_{t\to+\infty}f(\gamma(t))=c_f$ for $1$-a.e. $\gamma\in\Gamma^{+\infty}$ and for some constant $c_f\in\mathbb R$. By Proof of Corollary \ref{thm1.1}, the condition $f\in L^1(X)$ gives that $c_f=0$. Hence the assumption of ${\rm Mod}_1(\Gamma^{+\infty})>0$ implies the claim.
		
		$(1)\Longrightarrow \{ (2),(3),(5), (6)\}$ is obtained from Corollary \ref{thm1.1}. 
		
		$(1)\Longrightarrow (7)$  is obtained because \eqref{thm1.1-0711} and the assumption of ${\rm Mod}_1(\Gamma^{+\infty})>0$.
		
		$\{ (2),(3),(4),(5),(6),(7)\}\Longrightarrow(1)$ is given by Lemma \ref{lem4.1-1809} and since the claim that ${\rm Mod}_1(\Gamma^{+\infty}) \Longleftrightarrow \mathcal R_1<+\infty$ holds true by \cite{KN23}.
	\end{proof}

	\section{Examples}
	\begin{example}\label{ex3.3-0912}
		There exists an unbounded metric measure space $(X,d,\mu)$ so that the following statements hold true:
		\begin{enumerate}
			\item $\mu$ is not a doubling measure as defined in \eqref{eq2.4-0912} on $(X,d)$.
			\item Our space $(X,d,\mu)$ does not support a $1$-Poincar\'e inequality as defined in \eqref{eq2.5-2711}.
			\item There are an infinite curve $\gamma$, a function $h\in\dot {\rm BV}(X)$ and a function $f\in \dot{\rm BV}(X)$ such that
			\[
			{\rm AM}(\{\gamma\})\geq 1
			\]
			and
			\begin{equation}\label{equ3.10-0912}
				\begin{cases}
					\lim_{t\to+\infty}h^\vee(\gamma(t)) \text{\rm \ \ fails to exist,}\\
					h^\wedge(\gamma(t))\equiv 0 {\rm \ \ for\ all}\  t>0,\\
					\lim_{t\to+\infty}h^*(\gamma(t))\text{\rm \ \ exists.}
				\end{cases}
				{\rm \ \ and\ \ }
				\begin{cases}
					\lim_{t\to+\infty}f^\vee(\gamma(t)) \text{\rm \ \ fails to exist,}\\
					\lim_{t\to+\infty} f^\wedge(\gamma(t)) \text{\rm \ \ fails to exist,}\\
					\lim_{t\to+\infty}f^*(\gamma(t))\text{\rm \ \ exists.}
				\end{cases}
			\end{equation}

		\end{enumerate}
		
	\end{example}
	\begin{proof}
		\ 
		
		We first construct our space in $\mathbb R^2$. Let $X$ be a set defined by  $X:= X_1\cup X_2\cup X_3$ where
		\[
		X_1:= \{ (x,y)\in \mathbb R^2: y=1, x\ge 0\}, \ \ X_2:= \{ (x,y)\in \mathbb R^2: y=0, x\ge 0\}
		\]
		and
		\[
		X_3:=\{(x,y)\in\mathbb R^2: x\in\mathbb N, 0\leq y\leq 1\}.
		\]
		Roughly speaking, the space $X$ is a ladder with infinite steps. We equip $X$ with
		the length distance $d$ induced by the $2$-Euclidean distance and we equip $X$ with a measure $\mu$ defined by setting 
		\[
		\mu=\mathcal L^1|_{X_1\cup X_2}+\sum_{k\in\mathbb N}2^{-k}\mathcal L^1|_{I_k}
		\]
		where $\mathcal L^1$ is the $1$-Lebesgue measure and $I_k:= \{(x,y)\in X_3: x=k, 0\leq y\leq 1\}$. Here $\mathcal L^1|_A$ means the $1$-Lebesgue measure restricted to $A$.
		
		We first show that $\mu$ is not a doubling measure on $(X,d)$. For each $k\in\mathbb N$, we pick a point ${\rm x_k}:=(k,1/2)\in I_k$. Then it is easy to see that 
		\begin{equation}\notag
			\mu(B( {\rm x_k}, 1/2))= 2^{-k} \text{\rm \ \ and \ \ }\mu(B({\rm x_k}, 1))=2^{-k}+2
		\end{equation}
		for $k\in \mathbb N$.
		It follows that $\mu$ is not doubling because 
		\begin{equation}
			\notag
			\lim_{k\to+\infty} \frac{\mu(B({\rm x_k},1))}{\mu(B({\rm x_k},1/2))}= \lim_{k\to+\infty} \frac{2^{-k}+2}{2^{-k}}=+\infty.
		\end{equation}
		
		Now we show that our space does not support a $1$-Poincar\'e inequality. For each $k\in\mathbb N$, we consider again the point ${\rm x_k}:=(k,1/2)\in I_k$. It is clear that there is a $1$-Lipschitz function $f$ on $X$ so that 
		\[
		f\equiv 1 \text{\rm \ \ on $X_1$}, \ \ f\equiv 0 \text{\rm \ \ on $X_2$}, \ \ {\rm and \ \ }|\nabla f|=1 \text{\rm \ \ on $I_k$ for each $k\in\mathbb N$}.
		\]
		We now consider the sequence of balls $B_k:=B({\rm x}_k,1)$ for $k\in\mathbb N$.
		We see that $f_{B_k}=1/2$ for all $k\in\mathbb N$.
		Combining this with $\mu(B_k\cap X_1)= \mu(B_k\cap X_2)=1$ and $\mu(B_k)=2+2^{-k}$ for $k\in\mathbb N$, we obtain that
		\begin{equation}\label{eq3.10-0912}
			\dashint_{B_k}|f-f_{B_k}|d\mu
			\geq  \frac{1}{\mu(B_k)}  \left(\int_{B_k\cap X_1}|f-f_{B_k}|d\mu + \int_{B_k\cap X_2}|f-f_{B_k}|d\mu\right) 
			= \frac{1}{2+2^{-k}}
		\end{equation}
		for $k\in\mathbb N$.
		Moreover, we have that  for all given $\lambda\geq 1$ and for $k\in\mathbb N$,
		\begin{align}\label{eq3.11-0912}
			\dashint_{\lambda B_k}|\nabla f|d\mu= \frac{1}{\mu(\lambda B_k)}\sum_{j: I_j\subset\lambda B_k}\int_{I_j} d\mu \leq \frac{1}{\mu(B_k)} \mu(\lambda B_k\cap X_3)=\frac{\mu(\lambda B_k\cap X_3)}{2+2^{-k}}.
		\end{align}
		To obtain a $1$-Poincar\'e inequality \eqref{eq2.5-2711}, we need from \eqref{eq3.10-0912}-\eqref{eq3.11-0912} that there is a constant $C>0$ so that
		\[
		\frac{1}{2+2^{-k}}\leq \dashint_{B_k}|f-f_{B_k}|d\mu\leq C \dashint_{\lambda B_k}|\nabla f|d\mu \leq C \frac{\mu(\lambda B_k\cap X_3)}{2+2^{-k}}
		\]
		for all  $k\in\mathbb N$ and for some given $\lambda\geq 1$. Notice that for each given $\lambda\geq 1$ the right-hand side of this inequality converges to $0$ as $k\to+\infty$ because $\mu(X_3)<+\infty$. Letting $k\to+\infty$, we conclude that the above inequality fails and so our space does not support a $1$-Poincar\'e inequality.
		
		Finally, we show the third claim.  Let $\gamma:[0,+\infty)\to X$ be the infinite curve defined by $\gamma(t)=(t,0)$ for each $t\in[0,+\infty)$. Then $\gamma\subset X_2$. Let $\{g_i\}_{i\in\mathbb N}$ be an arbitrary admissible function for computing ${\rm AM}(\{ \gamma\})$.
		We have
		\begin{equation}\label{eq3.16-1612}
			1\leq \liminf_{i\to+\infty}\int_\gamma g_ids \leq \liminf_{i\to+\infty}\int_{X_2} g_id\mathcal L^1 \leq \liminf_{i\to+\infty} \|g_{i}\|_{L^1(X,\mu)}.
		\end{equation}
		As $g_i$ is arbitrary, we obtain that ${\rm AM}(\{\gamma\})\geq 1$.  We only need  to prove that there exist a function $h\in \dot{\rm BV}(X)$ and a function $f\in \dot{\rm BV}(X)$ so that the claim \eqref{equ3.10-0912} holds. We define $h$ by setting $$h=\chi_{X_1}+\sum_{k\in\mathbb N}\chi_{I_k}.$$ Then $h^\wedge(\gamma(t))\equiv 0$ for all $t>0$,
		$
		h^\vee(\gamma(k))=1
		$
		for each $k\in\mathbb N$,
		and 
		$
		h^\vee(\gamma(s))=0
		$
		for all $s\in\mathbb R_+\setminus \mathbb N$.
		Hence $\lim_{t\to+\infty}h^\vee(\gamma(t))$ fails to exist. It's clear that $h^*(\gamma(t))\equiv 0$ for all $t\in\mathbb R_{+}\setminus N$ and that 
		\[
		h^*(\gamma(k))=\lim_{r\to0}\dashint_{B((k,0),r)}hd\mu=\lim_{r\to0}\frac{2^{-k}r}{2^{-k}r+2r}=\frac{2^{-k}}{2^{-k}+2}
		\]
		and so $\lim_{t\to+\infty}h^*(\gamma(t))$ exists and is equal to $0$. Thus, $h$ satisfies \eqref{equ3.10-0912},
		we now prove that $h\in\dot{\rm BV}(X)$. Let $h_i$ be a sequence of functions defined by 
		\[
		h_i((x,y))=
		\begin{cases}
			0& {\rm if\ }(x,y)\in X_2,\\
			i d((x,y), X_2) & {\rm if\ }(x,y)\in X_3, 0\leq y\leq \frac{1}{i},\\
			1& {\rm otherwise \ }.
		\end{cases}
		\]
		Then $h_i\to h$ in $L^1_{\rm loc}(X)$ as $i\to+\infty$, and $|\nabla h_i|=i \chi_{ \{(x,y)\in X_3:0\leq y\leq \frac{1}{i} \}}$. Moreover, by a simple computation, we have that
		\[
		\int_{X}|\nabla h_i|d\mu= \sum_{k\in\mathbb N} \int_{I_k} i \chi_{\{x\in I_k: 0\leq y\leq \frac{1}{i} \}} 2^{-k}d\mathcal L^1= \sum_{k\in\mathbb N}2^{-k} =1.
		\]
		It follows from the definition of bounded variation that $h\in \dot {\rm BV}(X)$. Let 
		\[
		f=\chi_{X_1}+\sum_{k\in\mathbb N} \left(\chi_{I_{2k}}-\chi_{I_{2k+1}}\right).
		\]
		By the same argument as above to function $h$, one can verify that $f\in\dot{\rm BV}(X)$ satisfies \eqref{equ3.10-0912}.
		The proof completes.
	\end{proof}

	\begin{example}\label{ex3.4-1612}
		There exists an unbounded metric measure space $(X,d,\mu)$ so that the following statements hold true:
		\begin{enumerate}
			\item $\mu$ is not a doubling measure as defined in \eqref{eq2.4-0912} on $(X,d)$.
			\item Our space $(X,d,\mu)$ does not support a $1$-Poincar\'e inequality as defined in \eqref{eq2.5-2711}.
			\item There are an infinite curve $\gamma$, a function $h\in\dot{\rm BV}(X)$ and a function $f\in\dot{\rm BV}(X)$ such that 
			\[
			{\rm AM}(\{\gamma\})\geq 1
			\]
			and 
			
			\begin{equation}\label{equ3.18-1912}
				\begin{cases}
					\lim_{t\to+\infty}h^\vee(\gamma(t)) \text{\rm \ \ fails to exist,}\\
					h^\wedge(\gamma(t))\equiv 0 {\rm \ \ for\ all}\  t>0,\\
					\lim_{t\to+\infty}h^*(\gamma(t))\text{\rm \ \ fails to exist.}
				\end{cases}
				{\rm \ \ and\ \ }
				\begin{cases}
					\lim_{t\to+\infty}f^\vee(\gamma(t)) \text{\rm \ \ fails to exist,}\\
					\lim_{t\to+\infty} f^\wedge(\gamma(t)) \text{\rm \ \ fails to exist,}\\
					\lim_{t\to+\infty}f^*(\gamma(t))\text{\rm \ \ fails to exist.}
				\end{cases}
			\end{equation} 
		\end{enumerate}
	\end{example}
	\begin{proof}\ 
		
		We construct our space in $\mathbb R^2$ as follows. Let $N_k:=(2^k,1)$ and $S_k:=(2^k,0)$ where $k\in\mathbb N$. We now define a curve $\gamma_k$ connecting $N_k,S_k$. We pick a sequence of points $\{N_{k,i}\}_{i\in\mathbb N}$ defined by $N_{k,i}=(2^k, 2^{-(i-1)})$ for $k,i\in\mathbb N$. Then the 2-dimensional Euclidean distance between $N_{k,i}$ and $S_k$ is $2^{-(i-1)}$ and $N_{k,i}$ belongs to the segment $[N_k,S_k]$ for each $k,i$. For each $i\in\mathbb N$ and $k\in\mathbb N$, in the annulus $A_{k,i}:=B(S_k, 2^{-(i-1)})\setminus B(S_k,2^{-i})$ there is a bijective  rectifiable curve parameterized by arc-length $\gamma_{k,i}:[0,2^k2^{-i}]\to \mathbb R^2$, denoted $\gamma_{k,i}(t)=(\gamma_{k,i}^1(t), \gamma_{k,i}^2(t))\in\mathbb R^2$, so that  
		\[
		2^{-i}\leq \gamma_{k,i}^2(t)\leq 2^{-(i-1)},\ \ \gamma_{k,i}\subset A_{k,i},\ \ \gamma_{k,i}(0)=N_{k,i}, \ \ \gamma_{k,i}(2^k2^{-i})=N_{k,i+1}
		\]
		and
		\begin{equation}\label{equ3.19-1912}
			\int_{\gamma_{k,i}} ds=2^{k} 2^{-i}
		\end{equation}
		where the line integral over $\gamma_{k,i}$ is induced by the $2$-dimensional Euclidean distance. 
		Roughly speaking, $\gamma_{k,i}$ is a  ``long curve" in the annulus $A_{k,i}$ moving down from the northern point $N_{k,i}$ of ball $B(S_k, 2^{-(i-1)})$ to the northern point $N_{k,i+1}$ of the ball $B(S_k,2^{-i})$. We set
		\begin{equation}\label{equ3.20-1912}
			\gamma_{k}:= \tilde \gamma_{k,1}\cup \left(\cup_{i=2}^{+\infty}\gamma_{k,i}\right)
		\end{equation}
		where $\tilde \gamma_{k,1}$ is the segment $[N_{k,1},N_{k,2}]$. Here $N_{k,1}=N_k$.
		Then $\gamma_{k}$ is a curve connecting $N_k$ and $S_k$ with  length comparable to $2^k$. We define our space $X$ by $X:=X_1\cup X_2\cup X_3$ where 
		\[
		X_1:= \{ (x,y)\in \mathbb R^2: y=1, x\ge 0\}, \ \ X_2:= \{ (x,y)\in \mathbb R^2: y=0, x\ge 0\},
		\]
		and
		\[
		X_3:=\cup_{k\in\mathbb N}\bigg(\tilde \gamma_{k,1}\cup \left(\cup_{i=2}^{+\infty}\gamma_{k,i}\right)\bigg).
		\]
		We equip $X$ with the $2$-dimensional Euclidean  distance $d$  and a measure defined by setting
		\begin{equation}\label{equ3.21-1912}
			\mu=\mathcal L^1|_{X_1\cup X_2}+\sum_{k\in\mathbb N} 2^{-k}\left(\mathcal L^1|_{\tilde \gamma_{k,1}} +\sum_{i=2}^{+\infty}\mathcal L^1|_{\gamma_{k,i}}\right).
		\end{equation}
		We then have by \eqref{equ3.19-1912}-\eqref{equ3.20-1912} that
		\begin{equation}\label{equa3.22-1912}
			\mu(\gamma_k \cap B(S_k,2^{-j}))=\sum_{i=j+1}^{+\infty} \int_{\gamma_{k,i}}2^{-k}d\mathcal L^1=\sum_{i=j+1}^{+\infty}2^{-k}2^k 2^{-i} \approx 2^{-j}
		\end{equation}
		for $j\in\mathbb N$.
		Let  ${\rm x}_k:=(2^k, \frac{3}{4})$ be the middle point of the segment $\tilde \gamma_{k,1}$ for each $k\in\mathbb N$. Then 
		\[
		\mu(B({\rm x}_k,1/4))=2^{-(k+1)} {\rm \ \ and \ \ }\mu(B({\rm x}_k,1/2))\geq 1/2
		\]
		for $k\geq 2$.
		It follows that $\mu$ is not doubling because
		\[
		\lim_{k\to+\infty}\frac{\mu(B({\rm x}_k,1/2)) }{\mu(B({\rm x}_k,1/4)) }=+\infty.
		\]

		Let $\gamma:[0,+\infty)\to X$ be the infinite curve defined by $\gamma(t)=(t,0)$ for $t\in[0,+\infty)$. Then $\gamma\subset X_2$. 
		By \eqref{eq3.16-1612}, we have ${\rm AM}(\{\gamma\})\geq 1$. 
		
		We need to show that there are two functions $h\in\dot{\rm BV}(X)$ and $f\in\dot{\rm BV}(X)$ so that \eqref{equ3.18-1912} holds. Let $$h=\chi_{X_1\cup X_3}=\chi_{X_1}+\sum_{k\in\mathbb N}\chi_{\gamma_k}.$$  We have  that
		\begin{equation}\label{3.22-4th4}\notag
			h^*(S_k)=\limsup_{r\to0}\dashint_{B(S_k,r)} hd\mu\geq \limsup_{j\to+\infty}\frac{1}{\mu(B(S_k,2^{-j}))} \int_{B(S_k,2^{-j})}hd\mu=  \limsup_{j\to+\infty}\frac{ \mu(\gamma_{k}\cap B(S_k,2^{-j})) }{\mu(B(S_k,2^{-j}))} 
		\end{equation}
		and
		\begin{equation}\label{3.23-4th4}\notag
			\mu(B(S_k,2^{-j}))= \mu(\gamma_{k}\cap B(S_k,2^{-j})) + \mu (X_2\cap B(S_k,2^{-j})) \approx 2^{-j}\approx  \mu(\gamma_{k}\cap B(S_k,2^{-j}))
		\end{equation}
		by \eqref{equa3.22-1912}. We then obtain that $h^*(S_k)\gtrsim 1$ and so the limit $\lim_{t\to+\infty}h^*(\gamma(t))$ fails to exist because $h^*(\gamma(s))=0$ for each $s\notin \mathbb N$.
		It is also clear that $h^\wedge(\gamma(t))\equiv 0$ for all $t>0$ and that $h^\vee$ fails to exist since $h^\vee(\gamma(2^k))=h^\vee(S_k)=1$ for each $k\in\mathbb N$ and $h^\vee(\gamma(s))=0$ for each $s\notin \mathbb N$.
		
		We now show that $h\in\dot{\rm BV}(X)$. For each $i\in\mathbb N$, we define a sequence of $h_i$ by setting
		\
		\[
		h_i({\rm x}):= 
		\begin{cases}
			0&{\rm \ \ if\ }{\rm x}\in X_2,\\
			\int_{\gamma_{S_k,{\rm x}}}idt& {\rm \ \ if \ }{\rm x}\in \gamma_k: \int_{\gamma_{S_k,{\rm x}}}dt\leq \frac{1}{i} \ {\rm for\ }k\in\mathbb N,\\
			1 & {\rm \ \ otherwise}.
		\end{cases}
		\]
		Here $\gamma_{S_k,{\rm x}}\subset \gamma_k$ means the geodesic curve connecting $S_k$ and ${\rm x}$ for ${\rm x}\in \gamma_k$. Let ${\rm x}_{k,i}\in \gamma_{k}$ be so that
		\begin{equation}\label{eq3.25-1912}
			\int_{\gamma_{S_k,{\rm x}_{k,i}}}dt=\frac{1}{i}.
		\end{equation}
		Then $h_i$ is $i$-Lipschitz on each $\gamma_{S_{k,{\rm x}_{k,i}}}$ and is locally constant outside such curves.  It is clear that $h_i\to h$ in $L^{1}_{\rm loc}(X)$ as $i\to+\infty$. Furthermore, we have from \eqref{eq3.25-1912} and \eqref{equ3.21-1912} that
		\begin{equation}\label{3.25-4th4}
			\int_{X}g_{h_i}d\mu=\sum_{k\in\mathbb N} \int_{\gamma_{S_k,{\rm x}_{k,i}}} i d\mu =\sum_{k\in\mathbb N} i2^{-k}\frac{1}{i}=1.\notag
		\end{equation}
		It follows from the definition of bounded variation functions that $h\in \dot {\rm BV}(X)$. Let
		\[
		f=\chi_{X_1}+\sum_{k\in\mathbb N}\left(\chi_{\gamma_{2k}}-\chi_{\gamma_{2k+1}} \right).
		\]
		By a similar argument as above for the function $h$, one can verify that $f\in\dot{\rm BV}(X)$ satisfies \eqref{equ3.18-1912}.
		
		Finally, we note that our space does not support a $1$-Poincar\'e inequality. For all $\lambda\geq 1$, there is  an index $k_\lambda$ such that for all  $k_\lambda<k\in\mathbb N$, we have that
		the ball $B(S_k, 2\lambda)$ does not meet $\gamma_{k-1}, \gamma_{k+1}$.
		Now considering the balls $B(S_k,2)$ and the function $h$, it is straightforward to show that a
		$1$--Poincar\'e inequality does not hold.
		
		The proof completes.
	\end{proof}
	
	\begin{remark}\label{rem3.5-1612}
		Let $X=X_1\cup X_2\cup X_3$ with metric $d$ and measure $\mu$  defined as in Proof of Example \ref{ex3.3-0912}. Then  the doubling and Poincar\'e inequalities condition fails.  We pick $\gamma_1\equiv X_1$, $\gamma_2\equiv X_2$, and a $2$-Lipschitz function $f$ satisfying that $f\equiv 1$ on $X_1$, $f\equiv 0$ on $X_2$ and
		\[
		f((x,y))=
		\begin{cases}
			1&{\rm \ \ if \ }(x,y)\in I_k, \frac{3}{4}\leq y\leq 1,\\
			2d((x,y),X_2) &{\rm \ \ if \ } (x,y)\in I_k, \frac{1}{4}\leq y\leq \frac{3}{4},\\
			0&{\rm \ \ if \ }(x,y)\in I_k, y\leq \frac{1}{4}
		\end{cases}
		\]
		for $k\in\mathbb N$ where $X_3=\cup_{k\in\mathbb N}I_k$.
		Then one can check that $f\in \dot N^{1,1}(X)$ and 
		\begin{equation}\label{eq3.27-2012}\notag
			{\rm AM}(\{\gamma_1\})\geq 1, \ \ {\rm AM}(\{\gamma_2\})\geq 1
		\end{equation}
		and
		\begin{equation}\label{eq3.28-2012}\notag
			f\equiv f^\vee\equiv f^\wedge\equiv f^*\equiv 1 \text{\rm \ \ on $\gamma_1$} \ \ {\rm and }\ \ f\equiv f^\vee\equiv f^\wedge\equiv f^*\equiv 0 \text{\rm \ \ on $\gamma_2$}.
		\end{equation}
		Since our space has the annular chain condition at $O$ where $O:=(0,1)$, we then obtain that the uniqueness of three limits
		$$\lim_{t\to+\infty}f^\vee(\gamma(t)), \ \ \lim_{t\to+\infty}f^\wedge(\gamma(t)), \lim_{t\to+\infty}f^*(\gamma(t)) \text{\rm \ \ for \ $1$-a.e.}\  \gamma\in\Gamma^{+\infty}$$
		does not hold if the condition of doubling properties and $1$-Poincar\'e inequalities fails. 
		
	\end{remark}
	
	\begin{example}
		\label{ex3.5-8thMar}
		There exists an unbounded metric measure space $(X,d,\mu)$ so that the following statements hold true:
		\begin{enumerate}
			\item $\mu$ is not a doubling measure as defined in \eqref{eq2.4-0912} on $(X,d)$.
			\item Our space $(X,d,\mu)$  supports a $1$-Poincar\'e inequality as defined in \eqref{eq2.5-2711}.
			\item There are an infinite curve $\gamma$, a function $h\in\dot {\rm BV}(X)$ and  a function $f\in\dot {\rm BV}(X)$ such that
			\[
			{\rm AM}(\{\gamma\})\geq 1
			\]
			and
			\begin{equation}\label{3.23-8thMar}
				\begin{cases}
					\lim_{t\to+\infty}h^\vee(\gamma(t)) \text{\rm \ \ fails to exist,}\\
					\lim_{t\to+\infty}h^\wedge(\gamma(t)) \text{\rm \ \ fails to exist,}\\
					h^*(\gamma(t))\equiv 0 {\rm \ \ for \ all \ }t>0,
				\end{cases}
				{\rm \ \ and\ \ }
				\begin{cases}
					\lim_{t\to+\infty}f^\vee(\gamma(t)) \text{\rm \ \ fails to exist,}\\
					f^\wedge(\gamma(t))\equiv 0 {\rm \ \ for \ all \ }t>0,\\
					\lim_{t\to+\infty}f^*(\gamma(t))= 0.
				\end{cases}
			\end{equation} 
			
		\end{enumerate}
	\end{example}
	\begin{proof}\ 
		
		Let $X\subset \mathbb R^2$ be defined by setting $X=X_1\cup X_2$ where
		\[
		X_1=\{(x,y)\in\mathbb R^2: x\in\mathbb N, -1\le y\le 1 \} \ {\rm \ and\ }\ X_2=\{(x,y)\in\mathbb R^2: y=0, x\ge 0 \}.
		\]
		We equip $X$ with the length distance $d$ induced by the $2$-Euclidean distance and a measure $\mu$ defined by 
		\[
		\mu= \sum_{k\in\mathbb N}2^{-k}\mathcal L^1|_{I_k}+\mathcal L^1|_{X_2}
		\]
		where $\mathcal L^1$ is the $1$-Lebesgue measure and $I_k=\{(x,y)\in X_1: x=k, -1\le y\le 1\}$. By the arguments of Example \ref{ex3.3-0912}, we have that our space $(X,d,\mu)$ is not doubling. 
		
		Let $\gamma:[0,+\infty)\to X$ be the infinite curve defined by $\gamma(t)=(t,0)$ for $0\leq t<+\infty$. Then $\gamma\subset X_2$ and ${\rm AM}(\{ \gamma\})\geq 1$ by \eqref{eq3.16-1612}.
		
		Next, we will show that our space supports a $1$-Poincar\'e inequality. We consider two cases as below:
		\begin{enumerate}
			\item [\textbf{Case 1:}] Let $B((x,y),r)$ be an arbitrary ball in $X$ where $r<1$.
			
			Suppose that such $B((x,y),r)$ contains one point $(k,0)$ in $X_1\cap X_2$ where $k\in\mathbb N$. Then there are four segments $L_1,L_2,L_3,L_4$ in $B((x,y),r)$ defined by $L_1:=\{(z,0)\in B((x,y),r): z\leq k\}$, $L_2:=\{(z,0)\in B((x,y),r):z\geq k\}$, $L_3:=\{(k,t)\in B((x,y),r): 0\leq t\leq 1\}$, $L_4:=\{(k,t)\in B((x,y),r): -1\leq  t\leq 0\}$. Let $u$ be a continuous function on $X$ with  $|\nabla u|\in L^1_{\rm loc}(X)$. By the triangle inequality and since $B((x,y),r)=\cup_{i=1}^4L_i$, we have that
			\begin{align*}
				\dashint_{B((x,y),r)}|u-u_{B((x,y),r)}|d\mu&\leq 2 \dashint_{B((x,y),r)}|u((z,t))-u((k,0))|d\mu((z,t))\\
				&=\frac{2}{\mu(B(x,y),r)}\sum_{i=1}^4\int_{L_i} |u((z,t))-u((k,0))|d\mu((z,t))\\
				&\leq \frac{2}{\mu(B(x,y),r)}\sum_{i=1}^4\int_{L_i} \int_{[(k,0),(z,t)]_i}|\nabla u| ((z',t'))d((z',t'))\ d\mu((z,t))
			\end{align*}
			where $[(k,0),(z,t)]_i$ is the segment in $L_i$ connecting two points $(k,0)$ and $(z,t)$. We will denote $\mu((x,y))d((x,y)):=\mu((x,y))d\mathcal L^1((x,y)):=d\mu((x,y))$. By the Fubini's theorem and $\mu((z,t))=\mu((z',t'))$ for all points $(z,t), (z',t')\in L_i$, the above estimate gives that
			\begin{align}
				& \dashint_{B((x,y),r)}|u-u_{B((x,y),r)}|d\mu\notag\\
				\leq &\frac{2}{\mu(B(x,y),r)} \sum_{i=1}^4 \int_{L_i} |\nabla u|(z',t') \mu((z',t')) \left( \int_{L_i\setminus [(k,0),(z',t')]_i} d((z,t))\right) d((z',t'))\notag \\
				\leq & \frac{2}{\mu(B(x,y),r)} 2r\int_{B((x,y),r)}|\nabla u|d\mu \label{3.27-4th4}\notag
			\end{align}
			which is a $1$-Poincar\'e inequality for all balls $B((x,y),r)$ containing one point $(k,0)$ in $X_1\cap X_2$ where $k\in\mathbb N$ and $r<1$.
			If $B((x,y),r)$ does not contain any point $(k,0)$ in $X_1\cap X_2$ where $k\in\mathbb N$, then it's clear that we obtain a $1$-Poincar\'e inequality for such a ball. Thus we are done for this case.
			
			\item [\textbf{Case 2:}]  Let $B((x,y),r)$ be an arbitrary ball in $X$ where $r>1$. 
			
			Let $(x,y)$ be the center of such a ball.  We pick two points $(k_{\min},0)$ and $(k_{\max},0)$  so that $k_{\min}=\min\{ z\in \mathbb N: d((x,y),(z,0))\leq r+1\}$ and $k_{\max}=\max\{ z\in \mathbb N: d((x,y),(z,0))\leq r+1\}$. Let $\tilde B_r:= \{(x,y)\in X: k_{\min}\leq x\leq k_{\max} \}$. Then 
			\[
			\text{\rm $\mu(B((x,y),r))\approx \mu(\tilde B_r)$ and $B((x,y),r)\subset \tilde B_r=\gamma_r\cup \left( \cup_{k=k_{\min}+1}^{k_{\max}-1}I_k\right) $} 
			\]
			where $\gamma_r:=\{(x,0)\in X_2: k_{\min}\leq x\leq k_{\max} \}$. Let $u$ be a continuous function on $X$ with $|\nabla u|\in L^1_{\rm loc}(X)$. By the triangle inequality, we have that
			\begin{align}
				& \dashint_{B((x,y),r)}|u-u_{B((x,y),r)}|d\mu\notag \\
				\lesssim& 2 \dashint_{\tilde B_r}|u-u_{\tilde B_r}|d\mu\notag \\
				\lesssim& 4 \dashint_{\tilde B_r}|u-u((k_{\min},0))|d\mu\notag \\
				=& \frac{4}{\mu(\tilde B_r)}\left( \int_{\gamma_r}|u-u((k_{\min},0))|d\mu+ \sum_{k=k_{\min}+1}^{k_{\max}-1}\int_{I_k}|u-u((k_{\min},0))|d\mu\right)\notag \\
				\leq & \frac{4}{\mu(\tilde B_r)} \bigg(\int_{\gamma_r}|u-u((k_{\min},0))|d\mu+ \sum_{k=k_{\min}+1}^{k_{\max}-1} \int_{I_k}|u-u((k,0))|d\mu \notag\\
				&+ \sum_{k=k_{\min}+1}^{k_{\max}-1} |u((k,0))-u((k_{\min},0))|2^{-k+1} \bigg).\label{3.24-4th4}\notag
			\end{align}
			By the argument of \textbf{Case 1} and since 
			$\mu$ on each $\gamma_r, I_k$ is constant, the Fubini's theorem gives that for $r>1$,
			\[
			\int_{\gamma_r}|u-u((k_{\min},0))|d\mu\leq |k_{\max}-k_{\min}|\int_{\gamma_r}|\nabla u|d\mu\leq 4r \int_{\gamma_r}|\nabla u|d\mu 
			\]
			and
			\[
			\int_{I_k} |u-u((k,0))|d\mu\leq \int_{I_k}|\nabla u|d\mu\leq r \int_{I_k}|\nabla u|d\mu  {\rm \ \ for\ k_{\min}<k<k_{\max}}.
			\]
			Moreover, we have that for $r>1$,
			\begin{align*}
				\sum_{k=k_{\min}+1}^{k_{\max}-1} |u((k,0))-u((k_{\min},0))|2^{-k+1} \leq& \left(\sum_{k=k_{\min}+1}^{k_{\max}-1}  2^{-k+1} \right)\int_{\gamma_r}|\nabla u|ds \\
				=&\left(\sum_{k=k_{\min}+1}^{k_{\max}-1}  2^{-k+1} \right)\int_{\gamma_r}|\nabla u|d\mu\\
				\leq& 10 \int_{\gamma_r}|\nabla u|d\mu\\
				<& 10 r\int_{\gamma_r}|\nabla u|d\mu.
			\end{align*}
			Combining all above computations, we conclude that
			\begin{align*}
				\dashint_{B((x,y),r)}|u-u_{B((x,y),r)}|\lesssim r \dashint_{\tilde B_r}|\nabla u|d\mu.
			\end{align*}
			It's clear that there is $\lambda\geq 1$ so that $\mu(B((x,y),r))\approx \mu(\tilde B_r)\approx \lambda B((x,y),r)$ and $B((x,y),r)\subset \tilde B_r\subset \lambda B((x,y),r)$. Then this estimate implies a $1$-Poincar\'e inequality for all balls $B((x,y),r)$ with $r>1$.
		\end{enumerate}

		We now need to construct  $h\in\dot{\rm BV}(X)$ and $f\in \dot{\rm BV}(X)$  to satisfy \eqref{3.23-8thMar}. Let 
		\begin{equation}\label{3.24-9thMarch}
			h(x,y)=
			\begin{cases}
				1& {\rm \ \ if\ }(x,y)\in X_1 {\rm \ and\ }y>0,\\
				0&{\rm \ \ if\ }(x,y)\in X_2,\\
				-1& {\rm \ \ if\ }(x,y)\in X_1 {\rm \ and\ }y<0.
			\end{cases}
		\end{equation}
		Then for $k\in\mathbb N$, we have $h^\vee(\gamma(k))=h^\vee((k,0))=1$ and $h^\wedge(\gamma(k))=h^\wedge((k,0))=-1$. It follows that $\lim_{t\to+\infty}h^\vee(\gamma(t))$ and $\lim_{t\to+\infty}h^\wedge(\gamma(t))$ fail to exists because $h^\vee(\gamma(s))=h^\wedge(\gamma(s))=0$ for all $s\in\mathbb R_+\setminus \mathbb N$. Moreover, for each $k\in\mathbb N$,
		\begin{align*}
			h^*(\gamma(k))=&\lim_{r\to0}\dashint_{B((k,0),r)}hd\mu\\
			=&\lim_{r\to0}\frac{\mu(\{(x,y)\in I_{k}: 0<y<r \})-\mu(\{(x,y)\in I_{k}: -r<y<0 \})}{2^{-k}r+2r}=0
		\end{align*}
		because $\mu(\{(x,y)\in I_{k}: 0<y<r \})=\mu(\{(x,y)\in I_{k}: -r<y<0 \})=2^{-{k}}r$, and so $h^*(\gamma(t))\equiv 0$ for all $t>0$ since
		$h^*(\gamma(s))=0$ for all $s\in\mathbb R_+\setminus \mathbb N$.
		Let $h_i$ be a sequence of functions defined by 
		\[
		h_i((x,y))=
		\begin{cases}
			i d((x,y), X_2) & {\rm if\ }(x,y)\in X_1,  0\leq y\leq \frac{1}{i},\\
			0& {\rm if\ }(x,y)\in X_2,\\
			-i d((x,y), X_2) & {\rm if\ }(x,y)\in X_1, -\frac{1}{i}\leq y\leq 0,\\
			1& {\rm otherwise \ }.
		\end{cases}
		\]
		Then $h_i\to h$ in $L^1_{\rm loc}(X)$ as $i\to+\infty$ and $\int_X|\nabla h_i|d\mu=2$. Hence $h\in \dot{\rm BV}(X)$. 
		Let 
		\[
		f=\chi_{X_1}.
		\]
		By the same argument as above for the  function $h$, one can verify that $f\in\dot{\rm BV}(X)$ satisfies \eqref{3.23-8thMar}.
		The proof completes.
	\end{proof}

	\begin{example}\label{ex3.6-30th03}
		There exists an unbounded metric measure space $(X,d,\mu)$ so that the following statements hold true:
		\begin{enumerate}
			\item $\mu$ is a doubling measure as defined in \eqref{eq2.4-0912} on $(X,d)$.
			\item Our space $(X,d,\mu)$ does not  support a $1$-Poincar\'e inequality as defined in \eqref{eq2.5-2711}.
			\item There are an infinite curve $\gamma$ and three functions $f_1,f_2,f_3\in\dot {\rm BV}(X)$  such that
			\[
			{\rm AM}(\{\gamma\})\geq 1
			\]
			and
			\begin{equation}\label{3.25-30thMar}
				\begin{cases}
					\lim_{t\to+\infty}f_1^\vee(\gamma(t)) \text{\rm \ \ fails to exist,}\\
					\lim_{t\to+\infty}f_1^\wedge(\gamma(t)) \text{\rm \ \ fails to exist,}\\
					f_1^*(\gamma(t))\equiv 0 {\rm \ \ for \ all \ }t>0,
				\end{cases}
			\end{equation} 
			and
			\begin{equation}\label{3.26-30thMar}
				\begin{cases}
					\lim_{t\to+\infty}f_2^\vee(\gamma(t)) \text{\rm \ \ fails to exist,}\\
					f_2^\wedge(\gamma(t))\equiv 0  {\rm \ \ for \ all \ }t>0,\\
					\lim_{t\to+\infty}f_2^*(\gamma(t))\text{\rm \ \ fails to exist,}
				\end{cases}{\rm \ \ and\ \ }
				\begin{cases}
					\lim_{t\to+\infty}f_3^\vee(\gamma(t)) \text{\rm \ \ fails to exist,}\\
					\lim_{t\to+\infty}f_3^\wedge(\gamma(t)) \text{\rm \ \ fails to exist,}\\
					\lim_{t\to+\infty}f_3^*(\gamma(t)) \text{\rm \ \ fails to exist}.
				\end{cases}
			\end{equation} 
			
		\end{enumerate}
	\end{example}
	\begin{proof}
		\ 
		
		Let $I_k,J_k,$ where $ k\in\mathbb N,$ be subsets of $\mathbb R^2$ defined by
		\[
		I_k:=\{(k,y)\in\mathbb R^2: 0<y<1\}\setminus \left(\cup_{j\in\mathbb N}\{(k,2^{-j})\}\right)
		\]
		and
		\[
		J_k:=\{(k,y)\in\mathbb R^2: -1<y<0\}\setminus \left(\cup_{j\in\mathbb N}\{(k,-2^{-j})\}\right).
		\]
		Let $X_2=\{ (x,y)\in\mathbb R^2: y=0, x\ge  0\}$. We set 
		\[
		X:=X_2\cup \left(\cup_{k\in\mathbb N} (I_k\cup J_k) \right).
		\]
		We set $\gamma:[0,+\infty)\to X_2\subset X$ to be the infinite curve defined by $\gamma(t)=(t,0)$ for each $t\in[0,+\infty)$. We equip $X$ with the $2$-Euclidean distance $d$  and the $1$-Lebesgue measure $\mathcal L^1$. It is clear that the space is $1$-Ahlfors regular and so it is doubling. Moreover, since $I_k$ is disconnected, we obtain that $(X,d,\mathcal L^1)$ does not support any $1$-Poincar\'e inequality. As before, one can verify that ${\rm AM}(\{\gamma\})\geq 1$. Now, we only need to build three functions $f_1,f_2,f_3\in \dot{\rm BV}(X)$ satisfying \eqref{3.25-30thMar}-\eqref{3.26-30thMar}. We set
		\begin{equation}\label{eq3.27-30thMar}
			\begin{cases}
				f_1=\sum_{k\in\mathbb N} \left( \chi_{I_k} -\chi_{J_k} \right),\\
				f_2=\sum_{k\in\mathbb N}\left(\chi_{I_k}+\chi_{J_k} \right),\\
				f_3=\sum_{k\in\mathbb N}\left( \chi_{I_{2k}}-\chi_{I_{2k-1}}\right).
			\end{cases}
		\end{equation}
		Then for $k\in\mathbb N$,  we have that 
		\[
		\begin{cases}
			f_1^\vee(\gamma(k))=1, f_1^\wedge(\gamma(k))=-1, f_1^*(\gamma(k))= 0,\\
			f_2^\vee(\gamma(k))=1, f_2^\wedge({\gamma(k)})= 0, f_2^*(\gamma(k))=\frac{1}{2},\\
			f_3^\vee(\gamma(2k))=1, f_3^\wedge(\gamma(2k-1))=-1, f_3^*(\gamma(2k))=\frac{1}{4}.
		\end{cases}
		\]
		Hence $f_1,f_2,f_3$ satisfy \eqref{3.25-30thMar}-\eqref{3.26-30thMar} since $f_i^\vee(\gamma(t))=f_i^\wedge(\gamma(t))=f^*(\gamma(t))=0$ for $t\notin\mathbb N, i=1,2,3$. Moreover, these three functions belong to $\dot{\rm BV}(X)$ because $0$ is their minimal upper gradient. 
		The proof completes.
	\end{proof}
	
	In Example \ref{ex3.6-30th03}, $X$ is not connected. We now give a connected space in $\mathbb R^3$.
	\begin{example}\label{ex3.7-30thMar}
		There exists an unbounded metric measure space $(X,d,\mu)$ so that the following statements hold true:
		\begin{enumerate}
			\item $\mu$ is a doubling measure as defined in \eqref{eq2.4-0912} on $(X,d)$.
			\item Our space $(X,d,\mu)$ does not  support a $1$-Poincar\'e inequality as defined in \eqref{eq2.5-2711}.
			\item There are a family $\Gamma$ of infinite curves and three functions $f_1,f_2,f_3\in\dot {\rm BV}(X)$  such that
			\[
			{\rm AM}(\Gamma)\geq 1
			\]
			and such that for all $\gamma\in \Gamma$, we have
			\begin{equation}\label{3.27-30thMar}
				\begin{cases}
					\lim_{t\to+\infty}f_1^\vee(\gamma(t)) \text{\rm \ \ fails to exist,}\\
					\lim_{t\to+\infty}f_1^\wedge(\gamma(t)) \text{\rm \ \ fails to exist,}\\
					f_1^*(\gamma(t))\equiv 0 {\rm \ \ for \ all \ }t>0,
				\end{cases}
			\end{equation} 
			and
			\begin{equation}\label{3.28-30thMar}
				\begin{cases}
					\lim_{t\to+\infty}f_2^\vee(\gamma(t)) \text{\rm \ \ fails to exist,}\\
					f_2^\wedge(\gamma(t))\equiv 0  {\rm \ \ for \ all \ }t>0,\\
					\lim_{t\to+\infty}f_2^*(\gamma(t))\text{\rm \ \ fails to exist,}
				\end{cases}
				{\rm \ \ and\ \ }
				\begin{cases}
					\lim_{t\to+\infty}f_3^\vee(\gamma(t)) \text{\rm \ \ fails to exist,}\\
					\lim_{t\to+\infty}f_3^\wedge(\gamma(t)) \text{\rm \ \ fails to exist,}\\
					\lim_{t\to+\infty}f_3^*(\gamma(t)) \text{\rm \ \ fails to exist}.
				\end{cases}
			\end{equation} 
		\end{enumerate}
	\end{example}
	\begin{proof}
		\ 
		
		For $k\in\mathbb N$, we let $I_k,J_k\subset\mathbb R^3$ be defined by
		\[
		I_k:=\{(k,y,z)\in\mathbb R^3: 0\leq y \leq  1, 0\leq  z< 1\} \setminus  \left( \cup_{j\in\mathbb N} \{(k,y,2^{-j}): 0<y<1\} \right)  
		\]
		and
		\[
		J_k:=\{(k,y,z)\in\mathbb R^3: 0\leq y\leq  1, -1<  z\leq 0\} \setminus \left( \cup_{j\in\mathbb N} \{(k,y,-2^{-j}): 0<y<1\} \right).
		\]
		Let $X_2=\{(x,y,z)\in\mathbb R^3: z=0, x\ge 0, 0\leq y\leq 1\}$. We set 
		\[
		X:=X_2\cup \cup_{k\in\mathbb N}(I_k\cup J_k).
		\]
		We equip $X$ with the $3$-Euclidean distance $d$  and the $2$-Lebesgue measure $\mathcal L^2$. Then $\mathcal L^2$ is $2$-Ahlfors regular and so it is doubling. 
		Notice that  $I_k, J_k$ contains rectangles where two neighbour rectangles glue at two points. The $1$-modulus of all nonconstant rectifiable curves in $I_k, J_k$ that pass through a gluing point is zero, see for instance \cite[Corollary 5.3.11]{HKST15}. Let $u$ be a function on $R_1\cup R_2$ so that $u|_{R_1}=1, u|_{R_2}=0$ where $R_1,R_2$ are two open neighbouring rectangles in $I_k$. It follows that $0$ is the minimal $1$-weak upper gradient of $\tilde u$ where $\tilde u$ is the zero extension of $u$ to $X$. This gives that $X$ does not support any $1$-Poincar\'e inequality.
		
		Let $\Gamma$ be the family of all infinite curves lying in $X_2$. Then ${\rm AM}(\Gamma)\geq 1$ by the following estimate
		\[
		1\leq \int_0^1\liminf_{i\to+\infty}\int_{\gamma_y} g_idsdy\leq \liminf_{i\to+\infty}\int_0^1\int_{\gamma_y}g_idsdy \leq\liminf_{i\to+\infty}\int_Xg_id\mathcal L^2
		\]
		for all sequences of functions  $g_i:X\to[0,+\infty]$ with $\liminf_{i\to+\infty}\int_{\gamma_y}g_ids\geq 1$ for each infinite curve $\gamma_y$ defined by $\gamma_y(t):=(t,y,0)$ where $0\leq y\leq 1$ and $t\geq 0$.  Here the second inequality is given by Fatou's lemma.
		Let $f_1,f_2,f_3$ be defined as in \eqref{eq3.27-30thMar} w.r.t. our subsets  $I_k,J_k$. One can easily verify that $f_1,f_2,f_3\in\dot{\rm BV}(X)$ satisfy \eqref{3.27-30thMar}-\eqref{3.28-30thMar}, noticing again  that $0$ is the minimal $1$-weak upper gradient of each of them. The proof completes.
	\end{proof}

	\section*{Acknowledgements}
	We thank the anonymous referee for a careful reading and thoughtful
	comments on the paper.
	
	\vspace{1cm}
	
	
	\newcommand{\etalchar}[1]{$^{#1}$}

\end{document}